\documentclass[oneside,english]{amsart}
\usepackage[T1]{fontenc}
\usepackage[latin9]{inputenc}
\usepackage{textcomp}
\usepackage{mathrsfs}
\usepackage{amsbsy}
\usepackage{amstext}
\usepackage{amsthm}
\usepackage{amssymb}
\usepackage{graphicx}
\input xy
\xyoption{all}
\usepackage{fancybox}
\usepackage{array}
\usepackage{color}
\usepackage[usenames,dvipsnames,svgnames,table]{xcolor}

\makeatletter
\numberwithin{equation}{section}
\numberwithin{figure}{section}
 \theoremstyle{definition}
 \newtheorem*{defn*}{\protect\definitionname}
\theoremstyle{plain}
\newtheorem{thm}{\protect\theoremname}
  \theoremstyle{plain}
  \newtheorem{prop}[thm]{\protect\propositionname}
  \theoremstyle{plain}
  \newtheorem{lem}[thm]{\protect\lemmaname}
  \theoremstyle{plain}
  \newtheorem{cor}[thm]{\protect\corollaryname}
  \theoremstyle{plain}
  \newtheorem*{conjecture*}{\protect\conjecturename}
 \newtheorem*{thm*}{\protect\theoremname}

\usepackage[all]{xy}

\makeatother

\usepackage{babel}
  \providecommand{\conjecturename}{Conjecture}
  \providecommand{\corollaryname}{Corollary}
  \providecommand{\definitionname}{Definition}
  \providecommand{\lemmaname}{Lemma}
  \providecommand{\propositionname}{Proposition}
\providecommand{\theoremname}{Theorem}

\begin{document}

\title{Foliations on unitary Shimura varieties in positive characteristic}

\author{Ehud de Shalit and Eyal Z. Goren}

\date{July 9, 2017}



\address{Ehud de Shalit, Hebrew University of Jerusalem, Israel}

\address{ehud.deshalit@mail.huji.ac.il}

\address{Eyal Z. Goren, McGill University, Montr\'eal, Qu\'ebec, Canada}

\address{eyal.goren@mcgill.ca}

\global\long\def\End{\mathbf{\mathrm{End}}}

\global\long\def\Hom{\mathrm{Hom}}

\global\long\def\div{\mathrm{div}}

\global\long\def\Lie{\mathrm{Lie}}

\global\long\def\ord{\mathrm{ord}}

\global\long\def\no{\mathrm{no}}

\global\long\def\fol{\mathrm{fol}}

\global\long\def\Fr{\mathrm{Fr}}

\global\long\def\Ver{\mathrm{Ver}}

\global\long\def\Spec{\mathrm{Spec}}

\global\long\def\rk{\mathrm{rk}}
\begin{abstract}
When $p$ is inert in the quadratic imaginary field $E$ and $m<n$,
unitary Shimura varieties of signature $(n,m)$ and a hyperspecial
level subgroup at $p$, carry a natural \emph{foliation} \emph{of
height }1 and rank $m^{2}$ in the tangent bundle of their special
fiber $S$. We study this foliation and show that it acquires singularities
at deep Ekedahl-Oort strata, but that these singularities are resolved
if we pass to a natural smooth moduli problem $S^{\sharp}$, a successive
blow-up of $S.$ Over the ($\mu$-)ordinary locus we relate the foliation
to Moonen's generalized Serre-Tate coordinates. We study the quotient
of $S^{\sharp}$ by the foliation, and identify it as the Zariski
closure of the ordinary-\'etale locus in the special fiber $S_{0}(p)$
of a certain Shimura variety with parahoric level structure at $p$.
As a result, we get that this ``horizontal component'' of $S_{0}(p)$,
as well as its multiplicative counterpart, are non-singular (formerly
they were only known to be normal and Cohen-Macaulay). We study two
kinds of integral manifolds of the foliation: unitary Shimura subvarieties
of signature $(m,m)$, and a certain Ekedahl-Oort stratum that we
denote $S_{\fol}.$ We conjecture that these are the only integral
submanifolds.
\end{abstract}
\maketitle

\tableofcontents{}

\section{Introduction}

Inseparable morphisms in characteristic $p$ have long become an important
tool for obtaining deep results in algebraic geometry. The striking
difference between algebraic differential geometry in characteristic
0 and in characteristic $p$ is the source of many important applications.
We cite the proof by Rudakov and Shafarevich \cite{=00005BRu-Sh=00005D}
of the non-existence of global vector fields on a K3 surface, and
the characteristic $p$ proof by Deligne and Illusie \cite{=00005BDe-Ill=00005D}
of the degeneration of the Hodge spectral sequence, as two outstanding
examples.

The present paper originated from observations made in a special case
in \cite{=00005BdS-G2=00005D}. Its goal is to study the geometry
of unitary Shimura varieties modulo $p$, making use of the relation
between inseparable morphisms of height 1 and height 1 foliations
in the tangent bundle. This relation can be traced back, at the birational
level, to Jacobson's ``Galois theory'' for inseparable field extensions
by means of derivations \cite{=00005BJa=00005D}. It was further developed
by Rudakov and Shafarevich in \cite{=00005BRu-Sh=00005D}, and by
Ekedahl \cite{=00005BEk=00005D} and Miyaoka \cite{=00005BMi=00005D}.
The latter extended the notion of a foliation from first-order foliations
to any order, to deal with the problem of non-uniqueness of solutions
of linear differential equations in characteristic $p.$

\medskip{}

To explain our main results, let $E$ be a quadratic imaginary field
and $p$ an odd prime which is inert in $E$. Let $S_{K}$ be a Shimura
variety associated with a unitary group $\boldsymbol{G}_{/\mathbb{Q}}$,
split by $E,$ of signature $(n,m)$ ($0<m<n),$ and with an ad\`elic
level subgroup $K\subset\boldsymbol{G}(\mathbb{A})$. Then $S_{K}$
is a smooth quasi-projective variety of dimension $nm$ over $E$
(the reflex field). We assume that $K=K_{\infty}K_{p}K^{p}$ where
$K_{p}$ is hyperspecial maximal compact at $p$. Under this assumption,
Kottwitz \cite{=00005BKo=00005D} has defined a smooth integral model
$\mathcal{S}$ for $S_{K}$ over $\mathcal{O}_{E,(p)}$, whose special
fiber over the residue field $\kappa$ we denote simply by $S$. This
integral model $\mathcal{S}$ is a moduli space for certain $n+m$
dimensional abelian schemes with PEL structure (the endomorphisms
coming from $\mathcal{O}_{E}$) and we let $\mathcal{A}$ denote the
universal abelian scheme over it.

The special fiber $S$ admits a stratification by the isomorphism
type of the $p$-torsion of the abelian varieties making up the family
$\mathcal{A}.$ This is the Ekedahl-Oort (EO) stratification \cite{=00005BOo=00005D},
\cite{=00005BMo2=00005D}, \cite{=00005BWe=00005D}, see also \cite{=00005BV-W=00005D}.
It has a unique open dense stratum $S^{\ord},$ which coincides with
the unique open dense stratum in the Newton polygon (NP) stratification.
Under our assumption that $p$ is inert and $m<n$ the abelian varieties
parametrized by $S^{\ord}$ are not ordinary, but only $\mu$-ordinary.
This means, roughly speaking, that they are ``as ordinary as the
PEL data permits them to be''. This observation is where our journey
begins.

When $p$ is split, or when $p$ is inert but $n=m,$ $S^{\ord}$
classifies ordinary abelian varieties, in the usual sense. As Serre
and Tate have shown more than 50 years ago, a formal neighborhood
of a point $x\in S^{\ord}$ then carries a canonical structure of
a formal torus, hence $S^{\ord}$ is locally isotropic, i.e. ``looks
the same in all directions''. When $p$ is inert and $m<n$ this
breaks down. In fact, Moonen has introduced in \cite{=00005BMo1=00005D}
``generalized Serre-Tate coordinates'' on $S^{\ord}$, and showed
that under our assumptions, a formal neighborhood of a point $x\in S^{\ord}$
has a canonical structure as the part fixed under an involution in
a ``3-cascade''. Vasiu \cite{=00005BVa=00005D} has also done related
work, that we will not need to use in this paper. As a result of Moonen's
work, a canonical $m^{2}$-dimensional subspace $\mathcal{T}S_{x}^{+}$
is singled out in the tangent space $\mathcal{T}S_{x}$, and these
subspaces make up a sub-bundle $\mathcal{T}S^{+}$ over $S^{\ord},$
which is in fact a $p$-Lie sub-algebra, i.e. a height 1 foliation.

\medskip{}

One can define $\mathcal{T}S^{+}$ without any appeal to the results
of Moonen as follows. The Hodge bundle $\omega_{\mathcal{A}/S}=R^{0}\pi_{*}\Omega_{\mathcal{A}/S}^{1}$
breaks up as a direct sum $\mathcal{P}\oplus\mathcal{Q}$ according
to types, where $\mathcal{P}$, the part on which the endomorphisms
act via the natural map $\mathcal{O}_{E}\hookrightarrow\mathcal{O}_{E,(p)}\twoheadrightarrow\kappa$,
is of rank $n,$ and $\mathcal{Q},$ the part on which they act via
the Galois conjugate map, is of rank $m.$ The Kodaira-Spencer map
supplies us with an isomorphism
\[
KS:\mathcal{P}\otimes\mathcal{Q}\simeq\Omega_{S}^{1}.
\]

Once we situate ourselves in the special fiber, we can make use of
the Verschiebung isogeny $\Ver:\mathcal{A}^{(p)}\to\mathcal{A},$
and we denote by $V$ the map it induces on de~Rham cohomology. We
denote by $\mathcal{P}[V]$ the subsheaf which is the kernel of $V|_{\mathcal{P}}.$
Over $S^{\ord}$ it constitutes a sub-bundle $\mathcal{P}_{0}$ of
rank $n-m.$ We let
\[
\mathcal{T}S^{+}=KS(\mathcal{P}_{0}\otimes\mathcal{Q})^{\perp}
\]
be the annihilator of $KS(\mathcal{P}_{0}\otimes\mathcal{Q})$ in
$\mathcal{T}S$ under the natural pairing between the tangent and
cotangent bundles. The proof that this sub-bundle is in fact a foliation
(closed under Lie bracket and raising to power $p$) becomes a pleasant
exercise involving the Gauss-Manin connection and the notion of $p$-curvature
(see Proposition~\ref{prop:Being a foliation}). The foliation $\mathcal{T}S^{+}$
is the one appearing in the title of our paper. In Theorem~\ref{thm:compatibility with Moonen}
we prove that it coincides with what might be obtained from \cite{=00005BMo1=00005D}.

Besides its simplicity, our definition of the foliation has two other
advantages. First, according to the dictionary between foliations
of height 1 and inseparable morphisms of height 1, reviewed in \S\ref{subsec:Foliations-and-inseparable},
$\mathcal{T}S^{+}$ corresponds to a certain quotient variety of $S^{\ord}.$
In \S\ref{subsec:S_0(p) and the foliation} we identify this variety
as a Zariski open subset (\emph{the ordinary-\'etale locus}) in the
special fiber of an integral model of a certain Shimura variety of
parahoric level at $p.$ We call this Shimura variety $S_{K_{0}(p)}.$
In characteristic 0 it is a finite \'etale covering of $S_{K}.$ Its
integral model over $\mathcal{O}_{E,(p)}$, denoted $\mathcal{S}_{0}(p)$,
was defined by Rapoport and Zink in Chapter 6 of \cite{=00005BRa-Zi=00005D}
and studied further by several authors. For instance, G\"ortz \cite{=00005BG=0000F6=00005D}
proved that it is flat over $\mathcal{O}_{E,(p)},$ and that if we
denote its special fiber by $S_{0}(p),$ the local rings of the irreducible
components of $S_{0}(p)$ are Cohen-Macaulay and normal. See also
the work of Pappas and Zhu \cite{=00005BP-Z=00005D}. We make strong
use of these results later on. The special fiber $S_{0}(p)$ classifies
abelian schemes in characteristic $p,$ with PEL structure as in $S,$
equipped with a finite flat, isotropic, $\mathcal{O}_{E}$-stable
``Raynaud'' subgroup scheme $H$ of rank $p^{2m}.$ The ordinary-\'etale
locus of $S_{0}(p)$, denoted $S_{0}(p)_{et}^{\ord},$ is the open
subset lying over $S^{\ord}$ classifying such objects in which $H$
is \'etale. The alleged quotient map $S^{\ord}\to S_{0}(p)_{et}^{\ord}$
is such that when we compose it with the natural projection $S_{0}(p)_{et}^{\ord}\to S^{\ord}$,
in any order, we get the map $Fr_{p}^{2}$.

More important, perhaps, is that we are able to extend $\mathcal{T}S^{+}$
into the deeper EO strata, something absent from Moonen's theory of
generalized Serre-Tate coordinates. This is tied up with the study
of the closure $S_{0}(p)_{et}$ of the ordinary-\'etale locus $S_{0}(p)_{et}^{\ord},$
one of the horizontal relatively irreducible\footnote{A \emph{relatively irreducible }component $Y$ is a union of irreducible
components of $S_{0}(p)$ for which the projection $\pi:S_{0}(p)\to S$
induces a bijection between the irreducible components of $Y$ and
those of $\pi(Y).$ We call a relatively irreducible component \emph{horizontal}
if an open subset of it maps finite-flat to $S.$ The special fiber
has non-horizontal components too\emph{.}} components of $S_{0}(p).$ It is also tied up with a certain moduli-scheme
$S^{\sharp},$ special to characteristic $p,$ which is a ``successive
blow-up'' of $S$ at deep enough EO strata. There does not seem to be a natural way to lift
$S^{\sharp}$  to characteristic 0, and in any case the moduli problem which it represents lives in characteristic $p$ only.
In \S\ref{subsec:The-moduli-scheme S=000023}
we define $S^{\sharp}$ as a moduli problem and prove that it is representable
by a smooth scheme over $\kappa.$ We also determine the dimensions
of the fibers of the morphism $f:S^{\sharp}\to S,$ and the open set
$S_{\sharp}\subset S$ over which $f$ is an isomorphism. The set
$S_{\sharp}$ ``interpolates'' between $S^{\ord}$ and a unique
minimal EO stratum, of dimension $m^{2}$, contained in it, which
we call $S_{\fol}$. All this information is described in terms of
the combinatorics of $(n,m)$-shuffles in the symmetric group $\mathfrak{S}_{n+m}.$ 

The height 1 foliation $\mathcal{T}S^{+}$ extends canonically to
a height 1 foliation $\mathcal{T}S^{\sharp+}$ on $S^{\sharp}.$ Over
$S_{\sharp}$ it can be considered to lie in $S,$ but outside $S_{\sharp}$
it would acquire singularities, and it is necessary to introduce the
successive blow up $S^{\sharp}$ to extend it everywhere.

Having constructed $S_{0}(p)_{et}$ and $S^{\sharp}$, we describe
purely inseparable morphisms
\[
S_{0}(p)_{et}\overset{\pi_{et}^{\sharp}}{\to}S^{\sharp}\overset{\rho}{\to}S_{0}(p)_{et}
\]
whose composition is $Fr_{p}^{2}$. The map $\pi_{et}^{\sharp}$ extends
the natural projection $\pi_{et}$ from $S_{0}(p)_{et}^{\ord}$ to
$S^{\ord},$ and $\rho$ extends the quotient map $S^{\ord}\to S_{0}(p)_{et}^{\ord}$
obtained from the foliation $\mathcal{T}S^{+}.$ Using the theorem
of G\"ortz mentioned above we deduce that since $S_{0}(p)_{et}$ is
Cohen-Macaulay, $\pi_{et}^{\sharp}$ is finite and flat. Using the
normality of $S_{0}(p)_{et}$ we conclude that $\rho$ is also finite
and flat, and that $S_{0}(p)_{et}$ is in fact \emph{non-singular.}

At the other extreme we have the multiplicative horizontal component
$S_{0}(p)_{m},$ the closure of the ordinary-multiplicative locus
$S_{0}(p)_{m}^{\ord}$. It too maps to $S^{\sharp}$ and this map
is in fact an isomorphism, proving that $S_{0}(p)_{m}$ is non-singular.
Note that the projection from $S_{0}(p)_{m}$ to $S$ is \emph{not}
everywhere finite, and does not admit a section.

We do not know if similar results hold for the other, ``mixed''
horizontal components of $S_{0}(p),$ of which there are many in general.
We also stress that although $S_{0}(p)_{m}$ and $S_{0}(p)_{et}$
intersect, the maps that we have constructed from them to $S^{\sharp}$
do not agree on the intersection. This is manifested already in signature
$(2,1),$ see \cite{=00005BdS-G2=00005D}. 

We gather the above results in the following theorem. For more details see
Propositions \ref{prop:pi_m^sharp} and \ref {prop:extension_of_foliation}, and
Theorem \ref{thm:smoothness_S_0(p)_et}.

\begin{thm*}
(i) The foliation $\mathcal{T}S^{+}$ extends uniquely to a rank $m^{2}$ foliation $\mathcal{T}S^{\sharp +}$ on $S^{\sharp}.$

(ii) The quotient of $S^{\sharp}$ by $\mathcal{T}S^{\sharp +}$ is $S_{0}(p)_{et}$, which is non-singular.

(iii) $S_{0}(p)_m$ is isomorphic to $S^{\sharp}.$
\end{thm*}

\medskip{}

So far we said nothing about the behavior of the foliation at the
cusps. As mentioned above, $S,$ hence also $S^{\sharp},$ admit smooth
compactifications at the cusps. The Hodge bundle,
with its decomposition $\mathcal{P}\oplus \mathcal{Q}$, as well as $\mathcal{P}_{0}=\mathcal{P}[V]$,
extend as locally free sheaves to a suitable partial compactification of 
$S^{\ord}$ (a smooth toroidal compactification with the Zariski closure of the non-ordinary locus removed). 
The Kodaira-Spencer isomorphism, on the other
hand, acquires log poles along the boundary. See \cite{dS-G3}, \S3
for a detailed analysis, where the foliation is linked to the study of certain
differential operators on unitary modular forms. In the present work, we only need the open Shimura variety $S$.
\medskip{}

In the last section we turn our attention to integral subvarieties
of $\mathcal{T}S^{\sharp+}$ in $S_{\sharp}$. We obtain the following
result. See  Theorem \ref{thm:Sfol_integral_submanifold} for a more precise formulation.

\begin{thm*}
(i) Embedded Shimura varieties associated to $U(m,m)$, or to an inner form of $U(m,m)$, are 
integral subvarieties of $\mathcal{T}S^{\sharp+}$ .

(ii) The EO stratum $S_{\fol}$ is an integral subvariety of $\mathcal{T}S^{\sharp+}$.
\end{thm*}

While part (i) is easy, the proof of (ii) uses the \emph{canonical filtration of $\mathcal{A}[p]$
}over the EO stratum\emph{ $S_{\fol},$ } and requires some effort,
although the idea behind it is simple.

We end the paper with a discussion of a conjecture ``of Andr\'e-Oort
type'' that Shimura varieties of signature $(m,m)$, and the EO stratum
$S_{\fol}$, are the only global integral subvarieties of $\mathcal{T}S^{\sharp+}$.
Despite the fact that we do not know if the foliation lifts to a height
$h$ foliation for $h>1$ in any natural way, hence locally formally
the integral subvarieties are not unique, we believe that its global
nature makes the conjecture plausible.

\medskip{}

The results obtained in this paper generalize results that have been
obtained for Picard modular surfaces, associated with a unitary group
of signature $(2,1),$ in \cite{=00005BdS-G2=00005D}. On the other
hand it seems that with some extra effort they should generalize to
all Shimura varieties of PEL type.

\medskip{}

\textbf{Acknowledgments. }We would like to thank L. Illusie and C.
Liedtke for helpful discussions related to this work.

\section{Background}

\subsection{\label{subsec:Unitary-Shimura-varieties}Unitary Shimura varieties}

\subsubsection{The Shimura variety and its integral model at a good prime}

Let $E$ be a quadratic imaginary field, $0\le m\le n$ and $\Lambda=\mathcal{O}_{E}^{n+m}$,
equipped with the hermitian pairing
\[
(u,v)=\,^{t}\overline{u}\left(\begin{array}{ccc}
 &  & 1_{m}\\
 & 1_{n-m}\\
1_{m}
\end{array}\right)v.
\]
Here $1_{k}$ is the identity matrix of size $k$. Let $\delta$ be
a square root of the discriminant of $E,$ so that $\overline{\delta}=-\delta,$
and denote by $\text{Im\ensuremath{_{\delta}(z)=(z-\overline{z})/\delta}}$.
Then
\[
\left\langle u,v\right\rangle =\text{Im\ensuremath{_{\delta}(u,v)}}
\]
is $\mathbb{Q}$-bilinear, skew-symmetric, satisfies $\left\langle au,v\right\rangle =\left\langle u,\overline{a}v\right\rangle $,
and $\Lambda$ is self-dual, i.e. $\left\langle .,.\right\rangle $
induces $\Lambda\simeq\Hom(\Lambda,\mathbb{Z})$. Let $\boldsymbol{G}$
be the general unitary group of $(\Lambda,(.,.)),$ viewed as a group
scheme over $\mathbb{Z}.$ For every commutative ring $R$
\[
\boldsymbol{G}(R)=\left\{ g\in GL_{n+m}(\mathcal{O}_E\otimes R)|\,\exists\mu(g)\in R^{\times},\,\,(gu,gv)=\mu(g)(u,v)\right\} .
\]

Fix an \emph{odd} prime $p$ which is \emph{unramified} in $E$, and
an integer $N\ge3$ relatively prime to $p.$ Let $\mathbb{A}=\mathbb{R}\times\mathbb{A}_{f}$
be the ad\`ele ring of $\mathbb{\mathbb{Q}}$, where $\mathbb{A}_{f}=\mathbb{Q}\cdot\mathbb{\widehat{Z}}$
are the finite ad\`eles. Let $K_{f}\subset\boldsymbol{G}(\mathbb{\widehat{Z}})$
be an open subgroup of the form $K_{f}=K^{p}K_{p},$ where $K^{p}\subset\boldsymbol{G}(\mathbb{A}^{p})$
is the principal congruence subgroup of level $N$, and
\[
K_{p}=\boldsymbol{G}(\mathbb{Z}_{p})\subset\boldsymbol{G}(\mathbb{Q}_{p})
\]
the hyperspecial maximal compact subgroup at $p.$ Let $K_{\infty}\subset\boldsymbol{G}(\mathbb{R})$
be the stabilizer of the negative definite subspace spanned by $\{-e_{i}+e_{n+i};\,1\le i\le m\}$
in $\Lambda_{\mathbb{R}}=\mathbb{C}^{n+m}$, where $\left\{ e_{i}\right\} $
stands for the standard basis. This $K_{\infty}$ is a maximal compact-modulo-center
subgroup, isomorphic to $G(U(m)\times U(n)).$ By $G(U(m)\times U(n))$
we mean the pairs of matrices $(g_{1},g_{2})\in GU(m)\times GU(n)$
having the same similitude factor. Let $K=K_{\infty}K_{f}\subset\boldsymbol{G}(\mathbb{A})$
and $X=\boldsymbol{G}(\mathbb{R})/K_{\infty}.$

To the Shimura datum $(\boldsymbol{G},X)$ there is associated a Shimura
variety $Sh_{K}$. It is a quasi-projective smooth variety of dimension
$nm$ over $E$. If $m=n$ our Shimura variety is even defined over
$\mathbb{Q}$, but we still denote by $Sh_{K}$ its base-change to
$E$. The complex points of $Sh_{K}$ are identified, as a complex
manifold, with
\[
Sh_{K}(\mathbb{C})=\boldsymbol{G}(\mathbb{Q})\backslash\boldsymbol{G}(\mathbb{A})/K.
\]

Fix an embedding $\overline{\mathbb{Q}}\subset\overline{\mathbb{Q}}_{p}$
and let $v|p$ be the prime of $E$ induced by it. Following Kottwitz
\cite{=00005BKo=00005D}, but using a somewhat more restrictive set-up
suitable for principally polarized abelian varieties, we define a
scheme $\mathcal{S}$ over the localization $\mathcal{O}_{E,v}$ of
$\mathcal{O}_{E}$ at~$v$. This $\mathcal{S}$ is a fine moduli space
whose $R$-points, for every $\mathcal{O}_{E,v}$-algebra $R$, classify
isomorphism types of tuples $\underline{A}=(A,\iota,\phi,\eta)$ where
\begin{itemize}
\item $A$ is an abelian scheme of dimension $n+m$ over $R$.
\item $\iota:\mathcal{O}_{E}\hookrightarrow\End(A)$ has signature $(n,m)$
on the Lie algebra of $A$.
\item $\phi:A\overset{\sim}{\to}A^{t}$ is a principal polarization whose
Rosati involution induces $\iota(a)\mapsto\iota(\overline{a})$ on
the image of $\iota$.
\item $\eta$ is an $\mathcal{O}_{E}$-linear full level-$N$ structure
on $A$ compatible with $(\Lambda,\left\langle .,.\right\rangle )$
and $\phi$ (\cite{=00005BLan=00005D}, 1.3.6).
\end{itemize}
We shall summarize the above requirements by saying that $\underline{A}$
is a \emph{structure of type} $\mathscr{D}$ over $R$. See also \cite{=00005BLan=00005D}
for the comparison of the various languages used to define the moduli
problem. 

The generic fiber $S_{K}$ of $\mathcal{S}$ is, in general, a union
of \emph{several} Shimura varieties of the type $Sh_{K}.$ This is
due to the failure of the Hasse principle, which can happen when $m+n$
is odd. We also remark that the assumption $N\ge3$ could be avoided
if we were willing to use the language of stacks. As this is not essential
to the present paper, we keep the scope slightly limited for the sake
of clarity.

As shown by Kottwitz, $\mathcal{S}$ is \emph{smooth} of relative
dimension $nm$ over $\mathcal{O}_{E,v}.$ It even admits smooth (toroidal)
compactifications at the cusps, \emph{cf. }{[}Lan{]}.

\subsubsection{The universal abelian variety and the Kodaira-Spencer isomorphism}

The moduli space $\mathcal{S}$ carries a universal abelian scheme
$\mathcal{A}$ of dimension $n+m,$ equipped with $\iota,\phi,\eta$
as above. Let $\pi:\mathcal{A}\to\mathcal{S}$ be the structure morphism.
We denote by $\mathcal{A}^{t}$ the dual abelian scheme.

We let $\Sigma$ denote the identity embedding of $\mathcal{O}_{E}$
in $\mathcal{O}_{E,v}$ and $\overline{\Sigma}$ its complex conjugate.
Since $p$ is unramified in $E$, the locally free sheaves $H_{dR}^{1}(\mathcal{A}/\mathcal{S}),\,\,\omega_{\mathcal{A}/\mathcal{S}}=R^{0}$$\pi_{*}\Omega_{\mathcal{A}/\mathcal{S}}^{1}$
and $\omega_{\mathcal{A}^{t}/\mathcal{S}}^{\vee}=R^{1}\pi_{*}\mathcal{O}_{\mathcal{A}}$
decompose as direct sums of their $\Sigma$ and $\overline{\Sigma}$-parts
under the action of $\iota(\mathcal{O}_{E}).$ We write
\[
\mathcal{P}=\omega_{\mathcal{A}/\mathcal{S}}(\Sigma),\,\,\,\mathcal{Q}=\omega_{\mathcal{A}/\mathcal{S}}(\overline{\Sigma}).
\]
These are locally free sheaves of ranks $n$ and $m$ respectively
on $\mathcal{S}$.

The Kodaira-Spencer map is the \emph{sheaf homomorphism,}
\[
KS:\mathcal{P}=\omega_{\mathcal{A}/\mathcal{S}}(\Sigma)\to\Omega_{\mathcal{S}}^{1}\otimes\omega_{\mathcal{A}^{t}/\mathcal{S}}^{\vee}(\Sigma),
\]
obtained by embedding $\omega_{\mathcal{A}/\mathcal{S}}$
in $H_{dR}^{1}(\mathcal{A}/\mathcal{S}),$ applying the Gauss-Manin
connection
\[
\nabla:H_{dR}^{1}(\mathcal{A}/\mathcal{S})\to\Omega_{\mathcal{S}}^{1}\otimes H_{dR}^{1}(\mathcal{A}/\mathcal{S}),
\]
and finally projecting $H_{dR}^{1}(\mathcal{A}/\mathcal{S})$ to $\omega_{\mathcal{A}^{t}/\mathcal{S}}^{\vee}$.
Since the polarization $\phi$ induces an identification
\[
\phi^{*\vee}:\omega_{\mathcal{A}^{t}/\mathcal{S}}^{\vee}(\Sigma)\overset{\sim}{\leftarrow}\omega_{\mathcal{A}/\mathcal{S}}^{\vee}(\overline{\Sigma})=\mathcal{Q}^{\vee}
\]
the Kodaira-Spencer map yields a homomorphism, which we denote by
the same symbol
\[
KS:\mathcal{P}\otimes\mathcal{Q}\to\Omega_{\mathcal{S}}^{1}.
\]
This map turns out to be an \emph{isomorphism}. See \cite{=00005BLan=00005D}, Proposition 2.3.5.2.

\subsubsection{\label{subsec:The-NP-and EO}The NP and EO stratifications of the
special fiber of $\mathcal{S}$}

We briefly review some facts about these two stratifications, as the
EO stratification is going to play a central role later in the paper. 

Let $S$ be the special fiber of $\mathcal{S}.$ It is a smooth variety
over $\kappa=\kappa_{v},$ the residue field of $v$. Let $k$ be
an algebraically closed field containing $\kappa,$ and $x\in S(k).$
Let $NP_{x}$ be the Newton polygon of the $p$-divisible group $\mathcal{A}_{x}[p^{\infty}].$
It is lower convex, starts at $(0,0)$, ends at $(2(n+m),n+m)$, and
has integral break-points. Then $NP_{x}$ classifies the \emph{$k$-isogeny
class} of $\mathcal{A}_{x}[p^{\infty}].$ The set of Newton polygons
is partially ordered, where $P'\ge P$ if $P'$ lies on or above $P$.
For every Newton polygon $P$ there is a locally closed stratum $S_{P}$
in $S,$ defined over $\kappa,$ whose geometric points are precisely
those with $NP_{x}=P.$ The closure of a non-empty $S_{P}$ is the
union of the $S_{P'}$ for all $P'$ satisfying $P'\ge P$ (\cite{=00005BV-W=00005D},
\S11). This gives the Newton polygon (NP) stratification of $S.$

The Ekedahl-Oort (EO) stratification of $S$ is another stratification,
by the \emph{isomorphism type} of $\mathcal{A}_{x}[p]$. In addition
to the references already cited in the introduction, see also \cite{=00005BWoo=00005D}
for a thorough discussion of the case at hand. The EO strata $S_{w}$
are locally closed subsets labeled by certain elements $w$ in the
Weyl group $W$ of $\boldsymbol{G}.$ More precisely, the $w$ are
distinguished representatives for the cosets $W_{J}\backslash W$,
where the subgroup $W_{J}$ is determined by the signature condition.
The $S_{w}$ are equidimensional, smooth and quasi-affine. The dimension
of $S_{w}$ is $l(w),$ the length of $w$ relative to the Bruhat
order on $W$. The closure of $S_{w}$ is the union of $S_{w'}$ for
$w'\preceq w$ under a certain rather complicated order (related to,
but different from the usual Bruhat order; see \cite{=00005BV-W=00005D},
Theorems 2 and 3). We call it the \emph{EO order} on the Weyl group
elements indexing the strata. See below for a full description when
$p$ is inert in $E$ and $m<n$.

Wedhorn and Moonen have proved the following \cite{=00005BWe=00005D},
\cite{=00005BMo1=00005D}.

\medskip{}

\textbf{Fact. }There is a unique largest NP stratum, and a unique
largest EO stratum. These two strata coincide, and form an open dense
subset $S^{\ord}\subset S,$ called\footnote{In \cite{=00005BdS-G2=00005D} $S^{\ord}$ was denoted $S_{\mu}$
and called the $\mu$-ordinary locus.} the \emph{ordinary} locus of $S$. The isomorphism type of the whole
$p$-divisible group $\mathcal{A}_{x}[p^{\infty}]$ (with its endomorphisms
and polarization) is constant as $x$ varies along $S^{\ord},$ and
can be given explicitly in terms of the data $\mathscr{D}.$

\medskip{}

If $p$ is split in $E$ or $n=m$ then $\mathcal{A}_{x}[p^{\infty}]\simeq(\mathbb{Q}_{p}/\mathbb{Z}_{p})^{n+m}\times\mu_{p^{\infty}}^{n+m}$
for all $x\in S^{\ord}(k)$, so $\mathcal{A}_{x}$ is ordinary. If
$p$ is inert in $E$ and $m<n$ this is not the case, and for $x\in S^{\ord}$
\begin{equation}
\mathcal{A}_{x}[p^{\infty}]\simeq(\mathcal{O}_{E}\otimes\mathbb{Q}_{p}/\mathbb{Z}_{p})^{m}\times\mathscr{G}_{\Sigma}^{n-m}\times(\mathcal{O}_{E}\otimes\mu_{p^{\infty}})^{m}\label{eq:mu_ord_p_div_gp}
\end{equation}
where $\mathscr{G}_{\Sigma}$ is the unique 1-dimensional, height
2, slope 1/2, self-dual $p$-divisible group over $k$. The subscript
$\Sigma$ means that the embedding of $\mathcal{O}_{E}$ in $\End_{k}(\mathscr{G})$
via $\iota$ induces on $\Lie(\mathscr{G})$ the type $\Sigma,$ rather
than $\overline{\Sigma}.$ In this case it is customary to call $\mathcal{A}_{x}$,
for $x$ in the ordinary locus, $\mu$\emph{-ordinary}.

\[
\boxed{\text{From now on we assume that \ensuremath{p} is inert in \ensuremath{E} and \ensuremath{m<n.}}}
\]

\medskip{}

Under this assumption $v=(p),$ so we write $\mathcal{O}_{E,(p)}$
instead of $\mathcal{O}_{E,v},$ and $\kappa=\mathbb{F}_{p^{2}}.$
The lowest EO strata $S_{id}$ (labeled by $w=id$) is $0$-dimensional
and when $p$ is inert it classifies \emph{superspecial }abelian varieties,
i.e. those for which $\mathcal{A}_{x}[p^{\infty}]\simeq\mathscr{G}_{\Sigma}^{n}\times\mathscr{G}_{\overline{\Sigma}}^{m}.$
We call this stratum the \emph{core points.}

If $m=1$ the EO stratification has been worked out completely by
B\"ultel and Wedhorn \cite{=00005BB-W=00005D}. The strata are linearly
ordered, their dimensions dropping by 1 each time. Thus there are
$n+1$ EO strata altogether. As long as the dimension of the stratum
is strictly larger than $\left\lfloor n/2\right\rfloor $, the EO
strata are also NP strata. In fact, the isomorphism type of the whole
$p$-divisible group (with its endomorphisms and polarization) is
constant along these strata, as it was on $S^{\ord}.$ Half the way
through, in dimension $\left\lfloor n/2\right\rfloor $, one reaches
the supersingular NP stratum, which is stratified further by EO strata.
We remark also that from dimension $\left\lfloor n/2\right\rfloor $
down, the isomorphism type of $\mathcal{A}_{x}[p^{\infty}]$ is no
longer constant along the EO strata, only that of $\mathcal{A}_{x}[p].$

In general, we may identify the Weyl group of $\boldsymbol{G}$ with
$\mathfrak{S}_{n+m},$ the group of permutations of $\{1,\dots,n+m\}.$
Let $W_{J}=\mathfrak{S}_{n}\times\mathfrak{S}_{m}$. The elements
$w$ indexing the EO strata belong then to the set $\Pi(n,m)$ of
$(n,m)$-\emph{shuffles} in $\mathfrak{S}_{n+m}.$ A permutation $w$
is called an $(n,m)$-shuffle if
\[
w^{-1}(1)<\cdots<w^{-1}(n),\,\,\,\,w^{-1}(n+1)<\cdots<w^{-1}(n+m).
\]
The set $\Pi(n,m)$ is clearly a set of representatives for $W_{J}\backslash W.$

The dimension of $S_{w}$ is given by the formula
\begin{equation}
l(w)=\sum_{i=1}^{n}(w^{-1}(i)-i).\label{eq:length_of_w}
\end{equation}

We shall also need to know a formula for 
\[
a_{\Sigma}(w)=\dim\mathcal{P}_{x}[V]
\]
($x\in S_{w}(k)).$ Here $V$ is the map induced on cohomology by
Verschiebung, see below. This number is the $\Sigma$-part of Oort's
$a$-number of $\mathcal{A}_{x}.$ It turns out that it is given by
\begin{equation}
a_{\Sigma}(w)=|\{i|\,1\le i\le n,\,\,\,1\le w^{-1}(i)\le n\}|.\label{eq:sigma_part_a_number}
\end{equation}
For the formulae (\ref{eq:length_of_w}), (\ref{eq:sigma_part_a_number})
see \cite{=00005BWoo=00005D}, \S3.4 and \S3.5. For example, if $m=1,$
$a_{\Sigma}(w)=n-1$ except if $w=1$ (corresponding to the core points),
where it becomes $n.$ In general, for $w$ the longest $(n,m)$-shuffle
(of length $nm)$, where $S_{w}=S^{\ord},$ $a_{\Sigma}(w)=n-m.$
For $w=id$, corresponding to the core points, $a_{\Sigma}(id)=n.$

Finally, we make explicit the EO order relation $w'\preceq w$ on
the set $\Pi(n,m)$, following \cite{=00005BWoo=00005D}, Example
3.1.3. Recall that $w'\preceq w$ if and only if $S_{w'}\subset\overline{S_{w}}.$
Let
\[
w_{0,J}=\left(\begin{array}{cccccc}
1 & \dots & n & n+1 & \dots & n+m\\
n & \dots & 1 & n+m & \dots & n+1
\end{array}\right).
\]
Let $\le$ be the usual Bruhat order on $\mathfrak{S}_{n+m}$ with
respect to the standard set of reflections $s_{i}=(i,i+1)$ ($1\le i<n+m$).
Note that $w_{0,J}$ is the element of maximal length in $W_{J}.$
Then $w'\preceq w$ if and only if there exists a $y\in W_{J}$ such
that
\[
yw'w_{0,J}y^{-1}w_{0,J}\le w.
\]

Taking $y=1$ we see that if $w'\le w$ then also $w'\preceq w.$
This is the only property of the EO order relation that will be used
in the paper.

\subsubsection{\label{subsec:Frobenius,-Verschiebung-and}Frobenius, Verschiebung
and the Hasse invariant}

For any scheme $X$ in characteristic $p$ we denote by $\Phi_{X}$
the absolute Frobenius morphism of degree $p$ of $X$. Let
\[
\mathcal{A}^{(p)}=S\times_{\Phi_{S},S}\mathcal{A}
\]
be the base-change of the universal abelian variety. Let
\[
\Fr=Fr_{\mathcal{A}/S}:\mathcal{A}\to\mathcal{A}^{(p)}
\]
be the relative Frobenius morphism. It is an isogeny of abelian schemes
over $S,$ of degree $p^{n+m}.$ The isogeny dual to $Fr_{\mathcal{A}^{t}/S}$
is called the Verschiebung of $\mathcal{A}$ and is denoted
\[
\Ver=Ver_{\mathcal{A}/S}:\mathcal{A}^{(p)}\to\mathcal{A}.
\]
It too is of degree $p^{n+m}$ and $Ver_{\mathcal{A}/S}\circ Fr_{\mathcal{A}/S}$
is multiplication by $p$ on $\mathcal{A}.$

The maps induced by $Fr_{\mathcal{A}/S}$ and $Ver_{\mathcal{A}/S}$
on cohomology will be denoted $F$ and~$V.$ It is well-known that
\[
F:H_{dR}^{1}(\mathcal{A}^{(p)}/S)\to H_{dR}^{1}(\mathcal{A}/S)
\]
is a homomorphism of vector bundles of constant rank $\mathrm{rk}(\mathrm{Im}(F))=n+m.$ Its image
is a sub-bundle (i.e. a locally free sub-sheaf, the quotient of $H_{dR}^{1}(\mathcal{A}/S)$
by which is also locally free) and coincides with $H_{dR}^{1}(\mathcal{A}/S)[V].$
Similarly, the image of 
\[
V:H_{dR}^{1}(\mathcal{A}/S)\to H_{dR}^{1}(\mathcal{A}^{(p)}/S)
\]
is a sub-bundle, equal to $H_{dR}^{1}(\mathcal{A}^{(p)}/S)[F].$

The same can not be said about the restriction of $V$ to $\omega_{\mathcal{A}/S}.$
While $\omega_{\mathcal{A}/S}[V]$ is clearly a saturated sub-sheaf
of $\omega_{\mathcal{A}/S}$, its rank may increase when we move from
one EO stratum to a smaller one, contained in its closure. Hence,
$\omega_{\mathcal{A}/S}[V]$ is in general not a sub-bundle. It is,
however, a sub-bundle of rank $n-m$, if we restrict it to the
ordinary locus.

Since $F$ and $V$ commute with the endomorphisms, they induce maps
between the $\Sigma$ and the $\overline{\Sigma}$-parts. Note however
that
\[
H_{dR}^{1}(\mathcal{A}^{(p)}/S)(\Sigma)=H_{dR}^{1}(\mathcal{A}/S)^{(p)}(\Sigma)=H_{dR}^{1}(\mathcal{A}/S)(\overline{\Sigma})^{(p)}.
\]
In particular we get maps
\[
V_{\mathcal{P}}:\mathcal{P}\to\mathcal{Q}^{(p)},\,\,\,\,\,V_{\mathcal{Q}}:\mathcal{Q}\to\mathcal{P}^{(p)}.
\]
The sheaf homomorphism
\[
H_{\mathcal{A}/S}=V_{\mathcal{P}}^{(p)}\circ V_{\mathcal{Q}}:\mathcal{Q}\to\mathcal{Q}^{(p^{2})}
\]
is called the \emph{Hasse map}. Let $\mathcal{L}=\det(\mathcal{Q})$,
a line bundle on $S.$ Note that for every line bundle $\mathcal{L}$
there is a canonical isomorphism $\mathcal{L}^{(p)}\simeq\mathcal{L}^{p},$
sending $1\otimes s$ to $s\otimes\cdots\otimes s$ (here $s$ is
a section of $\mathcal{L}).$ Thus $h_{\mathcal{A}/S}=\det(H{}_{\mathcal{A}/S})$
is a homomorphism from $\mathcal{L}$ to $\mathcal{L}^{p^{2}},$ which
is the same as a global section
\[
h_{\mathcal{A}/S}\in H^{0}(S,\mathcal{L}^{p^{2}-1}),
\]
i.e. a modular form ``of weight $\mathcal{L}^{p^{2}-1}$'' called
the \emph{Hasse invariant}. This construction of the Hasse invariant, which plays an important role in the study
of $p$-adic modular forms, is due to Goldring and Nicole. See \cite[Appendix B]{GN}, and the main body of their paper for further generalizations. The relation with the stratifications
of $S$ is the following.

\medskip{}

\textbf{Fact. }\cite{=00005BWoo=00005D} Let $S^{\no}$ be the complement
of $S^{\ord}$ in $S$, endowed with its reduced subscheme structure.
Then $S^{\no}$ is a Cartier divisor and
\[
S^{\no}=\div(h{}_{\mathcal{A}/S}).
\]

\subsubsection{Pairings in de Rham cohomology of abelian varieties}

We review some general facts on de Rham cohomology of abelian varieties.
If $A/k$ is an abelian variety over a field $k$ (or more generally,
an abelian scheme over a ring) we let $A^{t}$ denote the dual abelian
variety. There is then a canonical perfect bilinear pairing
\[
\{.,.\}=\{.,.\}_{A}:H_{dR}^{1}(A/k)\times H_{dR}^{1}(A^{t}/k)\to k.
\]
When we use the canonical identification of $A$ with $(A^{t})^{t}$
we have
\[
\{u,v\}_{A}=-\{v,u\}_{A^{t}}.
\]
Indeed, one may identify $H_{dR}^{1}(A^{t}/k)$ with $H_{dR}^{2g-1}(A/k)$
($g=\dim A$), and then the pairing is given by cup product, followed
by the trace. If $\alpha:A\to B$ is an isogeny, we let $\alpha^{t}:B^{t}\to A^{t}$
be the dual isogeny, and then
\[
\{\alpha^{*}u,v\}_{A}=\{u,(\alpha^{t})^{*}v\}_{B}.
\]
If $\phi:A\overset{\sim}{\to}A^{t}$ is a principal polarization then
$\phi=\phi^{t}$ (using the identification of $(A^{t})^{t}$ with
$A$). The \emph{polarization pairing}
\[
\{u,v\}_{\phi}=\{u,(\phi^{-1})^{*}v\}_{A}:H_{dR}^{1}(A)\times H_{dR}^{1}(A)\to k
\]
is \emph{skew-symmetric}, as follows from the preceding two properties.

If $\alpha\in\End(A)$ then $Ros_{\phi}(\alpha)\in\End(A)$ is defined
by
\[
Ros_{\phi}(\alpha)=\phi^{-1}\circ\alpha^{t}\circ\phi.
\]
The previous properties imply then
\[
\{\alpha^{*}u,v\}_{\phi}=\{u,Ros_{\phi}(\alpha)^{*}v\}_{\phi}.
\]

If we apply the above for our $A$'s, figuring in a tuple $\underline{A}\in S(k)$,
we get that $\{,\}$ pairs $H_{dR}^{1}(A/k)(\Sigma)$ non trivially
with $H_{dR}^{1}(A^{t}/k)(\Sigma)$, while $\{,\}_{\phi}$ pairs $H_{dR}^{1}(A/k)(\Sigma)$
non-trivially with $H_{dR}^{1}(A/k)(\overline{\Sigma})$.

Finally we recall that $\omega_{A/k}$ and $\omega_{A^{t}/k}$ are
mutual annihilators of each other under $\{,\}.$ This induces a perfect
pairing between $\omega_{A/k}$ with $H^{1}(A^{t},\mathcal{O}),$
hence the identification of $H^{1}(A^{t},\mathcal{O})$ with the Lie
algebra of $A$.

\subsection{\label{subsec:Foliations-and-inseparable}Foliations and inseparable
morphisms of height 1}

\subsubsection{Foliations of height 1}

In this section we review some general facts from algebraic geometry
in characteristic $p,$ due to Rudakov and Shafarevich \cite{=00005BRu-Sh=00005D},
Ekedahl \cite{=00005BEk=00005D} and Miyaoka \cite{=00005BMi=00005D}.
At the birational level they should be traced back, as mentioned in
the introduction, to Jacobson's theorem which establishes a ``Galois
theory'' for finite purely inseparable field extensions using derivations
\cite{=00005BJa=00005D}.

Let $k$ be an algebraically closed field of characteristic $p,$
and $X$ a non-singular $n$-dimensional variety over $k.$ Let $\mathcal{T}X$
be the tangent sheaf of $X,$ a locally free sheaf of rank $n.$ Recall
that $\mathcal{T}X$ becomes a $p$-Lie algebra over $k$ if for any
two vector fields $\xi,\eta$ defined in some open set $U$ and regarded
as operators on $\mathcal{O}_{X}(U)$, we let
\[
[\xi,\eta]=\xi\circ\eta-\eta\circ\xi,\,\,\,\,\,\xi^{(p)}=\xi\circ\xi\circ\cdots\circ\xi
\]
(composition $p$ times).
\begin{defn*}
A \emph{foliation of height 1} on $X$ is a sub-bundle $\mathcal{E}\subset\mathcal{T}X$
(i.e. locally a direct summand), which is a $p$-Lie subalgebra, i.e.
involutive (closed under the Lie bracket) and closed under $\xi\mapsto\xi^{(p)}.$
\end{defn*}
Foliations of higher height (as in \cite{=00005BEk=00005D}) will
not show up in this paper, until we discuss integral subvarieties
at the end. We shall therefore refer to height 1 foliations simply
as ``foliations''. If $\mathcal{E}$ is a line sub-bundle then $\mathcal{E}$
is automatically involutive, as any two sections of $\mathcal{E}$
are proportional. But even in rank 1 the condition of being $p$-closed
is non-void, as the following example shows. Let $X=\mathbb{A}_{k}^{2}$
and let $\mathcal{E}=\mathcal{O}_{X}\cdot\xi$ where
\[
\xi=x\frac{\partial}{\partial x}+\frac{\partial}{\partial y}.
\]
It is easily checked that $\xi^{(p)}=x\frac{\partial}{\partial x}$,
but this is not a section of $\mathcal{E}.$

If $Y\subset X$ is a non-singular subvariety then $\mathcal{T}Y\subset\mathcal{T}X|_{Y}$
may be considered ``a foliation of height 1 along $Y$''. We call
$Y$ an \emph{integral subvariety }for the foliation $\mathcal{E}$
if $\mathcal{E}|_{Y}=\mathcal{T}Y$. Integral subvarieties always
exist in a formal neighborhood of a point (\cite{=00005BEk=00005D},
Proposition 3.2), but they need not be unique, even if they are global.
The foliation generated by the vector field $\partial/\partial y$
in $\mathbb{A}_{k}^{2}$ admits all the curves $x=a+by^{p}$ as integral
curves, and infinitely many such curves pass through any given point.
As another example, if $X$ is a simple abelian surface, then any
non-zero tangent vector at the origin generates a unique translation-invariant
line sub-bundle of $\mathcal{T}_X$. Often it can be shown to be $p$-closed, yielding a rank 1
foliation. This foliation has no integral curves at all, because an integral
curve would necessarily be an elliptic curve, contradicting the fact that
$X$ was assumed to be simple.

\subsubsection{The relation between foliations and inseparable morphisms of height
1}

Let $X$ be a non-singular $k$-variety. A finite $k$-morphism $X\overset{f}{\to}Y$
from $X$ to a $k$-variety $Y$ is called \emph{of height 1 }if there
is a $k$-morphism $Y\overset{g}{\to}X^{(p)}$ such that the composition
\[
X\overset{f}{\to}Y\overset{g}{\to}X^{(p)}
\]
is $Fr_{X/k}.$ Here $X^{(p)}=\Spec(k)\times_{\Phi_{k},\Spec(k)}X$.
If $f$ is also flat, then $Y$ is non-singular, since the property
of being regular descends under flat morphisms between locally Noetherian
schemes. In this case $f$ is faithfully flat, since it is surjective
on $k$-points, hence surjective. Therefore $g$ is also finite and
faithfully flat. Since $f^{(p)}\circ g\circ f=f^{(p)}\circ Fr_{X/k}=Fr_{Y/k}\circ f$
we get also $f^{(p)}\circ g=Fr_{Y/k}.$ The link between foliations
and height 1 morphisms is given by the following proposition.
\begin{prop}
\label{prop:=00005BEk=00005D,-Proposition-2.4.}\cite{=00005BEk=00005D},
Proposition 2.4. Let $X$ be a non-singular $k$-variety. There is
a natural 1-1 correspondence between finite flat height 1 morphisms
$f:X\to Y$ and height 1 foliations $\mathcal{E}\subset\mathcal{T}X.$
One has
\[
\deg(f)=p^{rk(\mathcal{E})}.
\]
Given $f$, $\Omega_{X/Y}$ is a locally free sheaf, hence the short
exact sequence
\[
f^{*}\Omega_{Y}\to\Omega_{X}\to\Omega_{X/Y}\to0
\]
splits and we let $\mathcal{E}=\Omega_{X/Y}^{\vee}\subset\Omega_{X}^{\vee}=\mathcal{T}X,$
a height 1 foliation. Conversely, given $\mathcal{E}$, we let $\mathcal{O}_{Y}=\mathcal{O}_{X}^{\mathcal{E}=0}$,
the sheaf of functions annihilated by the derivations in $\mathcal{E}.$
\end{prop}

The fact that $\mathcal{E}$ is a sub-bundle, and not only a saturated
subsheaf, is essential. Consider, for example, the subsheaf $\mathcal{E}$
of $\mathcal{T}\mathbb{A}^{2}$ generated by $\xi=x\cdot\partial/\partial x+y\cdot\partial/\partial y,$
which is saturated, but fails to be a sub-bundle at the origin. The
quotient of $\mathbb{A}^{2}$ by $\mathcal{E}$ is the scheme $Y$
for which $\mathcal{O}_{Y}$ is the sheaf of functions $h$ on $\mathbb{A}^{2}$
satisfying $\xi(h)=0$. This $Y$ is singular at the origin, and the
quotient map is not flat. Note that $h=x^{i}y^{p-i}$ is such a function
for $0\le i\le p$.

Let $f:X\to Y$ be as in the proposition, $x\in X$ and $y=f(x).$
Then one has the following local description of the completed local
rings.
\begin{prop}
\cite{=00005BEk=00005D}, Proposition 3.2. There is a system of formal
parameters $t_{1},\dots,t_{n}$ at $x$ such that $t_{1}^{p},\dots,t_{r}^{p},t_{r+1},\dots t_{n}$
is a system of formal parameters at the point $y$. In a formal neighborhood
of $x$ the foliation $\mathcal{E}$ is generated by $\partial/\partial t_{i}$
for $1\le i\le r.$
\end{prop}

\section{The foliation over the ordinary locus}

\subsection{Definition and first properties\label{subsec:Definition-and-first foliation}}

Let notation be as in \S\ref{subsec:Unitary-Shimura-varieties}. Over
$S^{\ord}$ the fibers of the subsheaf
\[
\mathcal{P}_{0}=\mathcal{P}[V]=\ker(V_{\mathcal{P}}:\mathcal{P}\to\mathcal{Q}^{(p)})
\]
have constant rank $n-m,$ hence, as $S^{\ord}$ is reduced, $\mathcal{P}_{0}$
is a sub-bundle of $\mathcal{P}.$ The sub-bundle $KS(\mathcal{P}_{0}\otimes\mathcal{Q})$
of $\Omega_{S}^{1}$ has accordingly rank $(n-m)m.$ We define a rank-$m^{2}$
sub-bundle $\mathcal{T}S^{+}\subset\mathcal{T}S^{\ord}$ by
\[
\mathcal{T}S^{+}=KS(\mathcal{P}_{0}\otimes\mathcal{Q})^{\perp}.
\]
\begin{prop}
\label{prop:Being a foliation}$\mathcal{T}S^{+}$ is a foliation
of height 1.
\end{prop}

The following two results will be used in the proof of the Proposition.
In the next lemma we use the identification
\[
\phi^{*}:\omega_{\mathcal{A}^{t}/S}(\Sigma)\simeq\omega_{\mathcal{A}/S}(\overline{\Sigma})=\mathcal{Q}
\]
induced by the polarization. Note that $\phi^{*}$ is type-reversing
because the Rosati involution induced by $\phi$ is $\iota(a)\mapsto\iota(\overline{a}).$
By $\nabla_{\xi}\alpha$ we denote, as usual, the contraction of $\nabla\alpha$
with the tangent vector $\xi$.
\begin{lem}
Let $\alpha\in\mathcal{P}=\omega_{\mathcal{A}/S}(\Sigma)$ and $\beta\in\mathcal{Q}\simeq\omega_{\mathcal{A}^{t}/S}(\Sigma).$
Denote by $\nabla$ the Gauss-Manin connection and by
\[
\{,\}:H_{dR}^{1}(\mathcal{A}/S)\times H_{dR}^{1}(\mathcal{A}^{t}/S)\to\mathcal{O}_{S}
\]
the canonical pairing in de Rham cohomology. Then for $\xi\in\mathcal{T}S$
we have
\[
\left\langle KS(\alpha\otimes\beta),\xi\right\rangle =\left\{ \nabla_{\xi}\alpha,\beta\right\} .
\]
\end{lem}

\begin{proof}
The lemma follows immediately from the definitions. Note that the
identification $R^{1}\pi_{*}\mathcal{O}_{\mathcal{A}}=\omega_{\mathcal{A}^{t}/S}^{\vee}$,
used in the definition of $KS$, results from the perfect pairing
$\{,\}$ and from the fact that under this pairing $\omega_{\mathcal{A}/S}$
and $\omega_{\mathcal{A}^{t}/S}$ are exact annihilators of each other.
\end{proof}
\begin{cor}
\label{cor: characterization of TS+}$\xi\in\mathcal{T}S^{+}$ if
and only if $\nabla_{\xi}(\mathcal{P}_{0})\subset\mathcal{P}_{0}.$
\end{cor}

\begin{proof}
If $\xi\in\mathcal{T}S^{+}$ then by the lemma $\nabla_{\xi}(\mathcal{P}_{0})$
is orthogonal under the pairing $\{,\}$ to $\mathcal{Q}\simeq\omega_{\mathcal{A}^{t}/S}(\Sigma).$
It is also orthogonal to $\omega_{\mathcal{A}^{t}/S}(\overline{\Sigma})$
for reasons of type. It is therefore orthogonal to the whole of $\omega_{\mathcal{A}^{t}/S},$
so $\nabla_{\xi}(\mathcal{P}_{0})\subset\omega_{\mathcal{A}/S}.$
But the Gauss-Manin connection commutes with isogenies and endomorphisms,
so $\nabla_{\xi}$ preserves the subspaces $H_{dR}^{1}(\mathcal{A}/S)(\Sigma)$
and $H_{dR}^{1}(\mathcal{A}/S)[V]$. It follows that
\[
\nabla_{\xi}(\mathcal{P}_{0})\subset\omega_{\mathcal{A}/S}(\Sigma)[V]=\mathcal{P}_{0}.
\]
The converse is obvious.
\end{proof}
We can now prove the proposition.
\begin{proof}
The involutivity of $\mathcal{T}S^{+}$ follows from the corollary,
since the Gauss-Manin connection is integrable, i.e.
\[
\nabla_{[\xi,\eta]}=\nabla_{\xi}\circ\nabla_{\eta}-\nabla_{\eta}\circ\nabla_{\xi}.
\]
The fact that $\mathcal{T}S^{+}$ is closed under $\xi\mapsto\xi^{(p)}$
is more subtle as the $p$-curvature
\[
\psi(\xi):=\nabla_{\xi^{(p)}}-\nabla_{\xi}^{(p)}
\]
does not vanish identically, but is only a nilpotent endomorphism
of $H_{dR}^{1}(\mathcal{A}/S)$ (\cite{=00005BKa=00005D}, Theorem
5.10). However, on the sub-module with connection $H_{dR}^{1}(\mathcal{A}/S)[V]$
the $p$-curvature vanishes. This is because the kernel of $V$ is
the image of $F$, so (the easy direction of) Cartier's theorem (\cite{=00005BKa=00005D},
Theorem 5.1) implies that $\psi(\xi)=0$ there. Since $\mathcal{P}_{0}\subset H_{dR}^{1}(\mathcal{A}/S)[V]$
we can conclude the proof as before, using the corollary.
\end{proof}

\subsection{\label{subsec:S_0(p) and the foliation}The Shimura variety of parahoric
level structure}

By Proposition \ref{prop:=00005BEk=00005D,-Proposition-2.4.}, the
height 1 foliation $\mathcal{T}S^{+}$ on $S^{\ord}$ corresponds
to a finite flat purely inseparable quotient of $S^{\ord}.$ Our purpose
in this section is to identify this quotient as the ordinary-\'etale
component of (the special fiber of) a certain Shimura variety of parahoric
level structure. This will allow us in \S\ref{sec:Extending-the-ordinary}
to extend the foliation to the non-ordinary locus.

\subsubsection{The Shimura variety $S_{K_{0}(p)}$}

In addition to the lattice $\Lambda=\mathcal{O}_{E}^{n+m}$ considered
in \S\ref{subsec:Unitary-Shimura-varieties}, consider also the $\mathcal{O}_{E}$-lattices
\[
\Lambda\supset\Lambda'\supset\Lambda''\supset p\Lambda
\]
where 
\[
\Lambda'=\left\langle pe_{1},\dots pe_{m},e_{m+1},\dots,e_{n+m}\right\rangle ,\,\,\,\,\,\,\Lambda''=\left\langle pe_{1},\dots pe_{n},e_{n+1},\dots,e_{n+m}\right\rangle .
\]
Note that the dual of $\Lambda'$ is $p^{-1}\Lambda''.$ Let $\mathscr{L}$
be the lattice chain in $K_{p}^{n+m}$
\[
\cdots\supset\Lambda_{p}\supset\Lambda_{p}'\supset\Lambda_{p}''\supset p\Lambda_{p}\supset\cdots
\]
obtained by tensoring with $\mathbb{Z}_{p}$ and extending by periodicity,
and let $K_{0}(p)_{p}$ be its stabilizer in $K_{p}=\boldsymbol{G}(\mathbb{Z}_{p}).$
This is a parahoric subgroup, and if we let $K_{0}(p)$ be the ad\`elic
level subgroup corresponding to it (and to full level $N$ as usual),
we get the Shimura variety $S_{K_{0}(p)}$, which is again defined
over $E$, and is an \'etale cover of $S_{K}$.

\subsubsection{The moduli problem $\mathcal{S}_{0}(p)$}

Let $\mathcal{S}_{0}(p)$ be the integral model of $S_{K_{0}(p)}$
over $\mathcal{O}_{E,(p)}$ which was constructed by Rapoport and
Zink in \cite{=00005BRa-Zi=00005D}, \S6.9. We want to give a more
concrete description of the moduli problem parametrized by $\mathcal{S}_{0}(p)$.
Let $R$ be an $\mathcal{O}_{E,(p)}$-algebra and $\underline{A}\in\mathcal{S}(R)$
as in \S\ref{subsec:Unitary-Shimura-varieties}. A finite flat $\mathcal{O}_{E}$-subgroup
scheme $H\subset A[p]$ is called \emph{Raynaud} if for every characteristic
$p$ geometric point $x:R\to k$ of $\Spec(R),$ the Dieudonn\'e module
$M(H_{x})$ is balanced, in the sense that
\[
\dim_{k}M(H_{x})(\Sigma)=\dim_{k}M(H_{x})(\overline{\Sigma}).
\]
By $M(H_{x})$ we denote the covariant Dieudonn\'e module of $H_{x}.$
It coincides with the contravariant Dieudonn\'e module of the Cartier
dual $H_{x}^{D}$ of $H_{x}$. See \cite{=00005BdS-G2=00005D} \S1.2.1
for a discussion of the Raynaud condition when $m=1$, and its relation to the
original condition imposed by Raynaud in \cite{=00005BRay=00005D}.
\begin{prop}
The scheme $\mathcal{S}_{0}(p)$ is a moduli space for pairs $(\underline{A},H)$
where $\underline{A}\in\mathcal{S}(R)$ and $H\subset A[p]$ is a
finite flat group scheme of rank $p^{2m}$, which is isotropic for
the Weil pairing on $A[p]$ induced by $\phi$, $\mathcal{O}_{E}$-stable
and Raynaud.
\end{prop}

\begin{proof}
See \cite{=00005BdS-G2=00005D}, \S1.3. The proof given there for $(n,m)=(2,1)$
can be adapted to the general case \emph{mutatis mutandis}.
\end{proof}
We denote\footnote{No confusion should arise from the fact that we denote by $\pi$ also
the structure map $\mathcal{A}\to\mathcal{S}$, or the ratio of the
circumference of the circle to its diameter.} by $\pi:\mathcal{S}_{0}(p)\to\mathcal{S}$ the morphism which on
the moduli problem is ``forget $H$''. The scheme $\mathcal{S}_{0}(p)$
is not smooth over $\mathcal{O}_{E,(p)}$ but G\"ortz \cite{=00005BG=0000F6=00005D},
and later Pappas and Zhu \cite{=00005BP-Z=00005D}, proved the following.
\begin{prop}
\label{prop:=00005BG=0000F6=00005D}(\cite{=00005BP-Z=00005D}, Theorem
0.2) The scheme $\mathcal{S}_{0}(p)$ is proper and flat over $\mathcal{O}_{E,(p)}$,
the irreducible components of its special fiber are reduced, and their
local rings are Cohen-Macaulay and normal.
\end{prop}

\subsubsection{The ordinary-multiplicative and ordinary-\'etale loci}

Let $S_{0}(p)$ be the special fiber of $\mathcal{S}_{0}(p)$. Let
$S_{0}(p)^{\ord}=\pi^{-1}(S^{\ord}).$ If $x\in S^{\ord}(k)$ is a
geometric point, then $\mathcal{A}_{x}[p^{\infty}]$ is given by (\ref{eq:mu_ord_p_div_gp}),
and any isotropic Raynaud $\mathcal{O}_{E}$-subgroup scheme of $\mathcal{A}_{x}[p]$
of rank $p^{2m}$ is of the form
\[
H\simeq(\mathcal{O}_{E}\otimes\mu_{p})^{r}\times\mathscr{G}_{\Sigma}[p]^{s}\times(\mathcal{O}_{E}\otimes\mathbb{Z}/p\mathbb{Z})^{t}
\]
for an ``ordinary type'' ($r,s,t),$ $r+s+t=m.$ The type $(r,s,t)$
is locally constant on $S_{0}(p)^{\ord}$ in the Zariski topology.
Indeed, $p^{2r+2s}$ is the rank of the connected part $H^{conn}$,
which can only go up under specialization, by duality the same is
true of $p^{2s+2t},$ but $r+s+t$ is constant. Thus under specialization
the only possibility is for $r$ and $t$ to go down, and for $s$
to go up. But the same must be true of $A[p]/H.$ If we specialize
to a $\mu$-ordinary point, the type of $A[p]/H$ is $(m-r,n-m-s,m-t)$,
hence $m-r$ and $m-t$ must also go down, forcing $r,s$ and $t$
to remain constant.

It follows that the discrete invariants $(r,s,t)$ decompose $S_{0}(p)^{\ord}$
into disjoint open sets. We denote by $S_{0}(p)_{m}^{\ord}$ the locus
where $H$ is of multiplicative type ($s=t=0$) and by $S_{0}(p)_{et}^{\ord}$
the locus where $H$ is \'etale ($r=s=0$). The other loci are denoted
by $S_{0}(p)_{r,s,t}^{\ord},$ and the projection from them to $S^{\ord}$
will be denoted $\pi_{r,s,t}.$
\begin{prop}
\label{prop:S_0(p)^ord_m,et are smooth}The loci $S_{0}(p)_{m}^{\ord}$
and $S_{0}(p)_{et}^{\ord}$ are non-singular and relatively irreducible
over $S^{\ord}$.
\end{prop}

\begin{proof}
This must be well-known. We shall see below that $S_{0}(p)_{m}^{\ord}$
is isomorphic to $S^{\ord}.$ The relative irreducibility of $S_{0}(p)_{et}^{\ord}$
is a consequence of the surjectivity of the map $\rho:S^{\ord}\to S_{0}(p)_{et}^{\ord}$
constructed in \S\ref{subsec:Morphisms-between-ordinary-=0000E9tale}.
Regularity can be proven by direct computation of the tangent space
using deformation theory, as outlined in \cite{=00005BdS-G2=00005D}
when $(n,m)=(2,1)$ (following Bella\"iche). Alternatively, one can
argue as follows. The input leading to the computation of the tangent
space (or even the local model) at a closed point $x\in S_{0}(p)_{et}^{\ord}$
is everywhere the same. This is because the $p$-divisible group is
$\mu$-ordinary and $H$ is the kernel of $p$ in its \'etale part.
Thus either all points of $S_{0}(p)_{et}^{\ord}$ are non-singular
or all are singular. The latter case is excluded since it is reduced,
by Proposition \ref{prop:=00005BG=0000F6=00005D}.
\end{proof}
Denote by $\pi_{m}$ and $\pi_{et}$ the restriction of the map $\pi$
to $S_{m}^{\ord}$ and $S_{et}^{\ord}$ (i.e. the maps $\pi_{m,0,0}$
and $\pi_{0,0,m}).$ The map $\pi_{m}$ is an \emph{isomorphism}.
In fact, it has the section associating to every $\underline{A}\in S^{\ord}(R)$
the $R$-point $(\underline{A},A[p]^{mult})\in S_{0}(p)_{m}^{\ord}(R)$
where $A[p]^{mult}$ is the maximal finite flat subgroup scheme of
$A[p]$ which is of multiplicative type (connected with \'etale dual).
This subgroup is automatically isotropic and Raynaud. Denote this
section by 
\[
\sigma_{m}:S^{\ord}\to S_{0}(p)_{m}^{\ord}.
\]

\subsubsection{Morphisms between ordinary-\'etale components\label{subsec:Morphisms-between-ordinary-=0000E9tale}}

We shall define morphisms $\theta,\theta',$ $\rho$ and $\rho'$
that fit into the commutative diagram below. The scheme $(S_{0}(p)_{et}^{\ord})^{(p)}$
appearing in the diagram classifies the same objects as $S_{0}(p)_{et}^{\ord},$
except that the signature of the $\mathcal{O}_{E}$ action is $(m,n)$
instead of $(n,m).$

\[\xymatrix{& (S_0(p)_{et}^{\text{ord}})^{(p)}\ar[d]^{\rho^\prime} \ar[dr]^{Fr_p} & \\ S_0(p)_{et}^{\text{ord}}\ar[ru]^{Fr_p}\ar[r]^(0.6){\pi_{et}}\ar[rd]_{\theta^\prime} & S^{\text{ord}}\ar@<-0.5ex>[d]_{\sigma_m}\ar[r]^(0.4)\rho & S_0(p)_{et}^{\text{ord}} \\ &  S_0(p)_{m}^{\text{ord}} \ar@<-0.5ex>[u]_{\pi_m}\ar[ur]_\theta }\]

\medskip{}

Before we embark on the definition, we want to make a notational remark.

\medskip{}

\textbf{Remark about Frobenii}: We denote by $Fr_{p}$ the Frobenius
of the scheme $S$ (or $S_{0}(p))$ relative to $\kappa.$ Thus, $Fr_{p}$
is a morphism of $\kappa$-schemes
\[
Fr_{p}:S\to S^{(p)},
\]
where $S^{(p)}=\Phi_{\kappa}^{*}S$ is the base change of $S$ with
respect to $\Phi_{\kappa}.$ If $\xi\in S(R)$ for some $\kappa$-algebra
$R$ and $\xi$ corresponds to the tuple $\underline{A}=\xi^{*}\underline{\mathcal{A}},$
then $Fr_{p}(\xi)\in S^{(p)}(R)$ corresponds to $\underline{A}^{(p)}=\Phi_{R}^{*}\underline{A},$
the base change of $A$ (with the associated PEL structure) \emph{with
respect to} $\Phi_{R}$. Note that $A^{(p)}$ has type $(m,n).$ We
write in short
\[
Fr_{p}(\underline{A})=\underline{A}^{(p)}.
\]
This should not be confused with the isogeny $\Fr:A\to A^{(p)}$,
which is a morphism of abelian schemes over $\Spec(R).$ A more appropriate
convention would have been to denote $Fr_{p}$ by $Fr_{S/\kappa}$
and the isogeny $\Fr$ by $Fr_{A/R}$ or $Fr_{\mathcal{A}/S},$ but
this would result in a pretty heavy notation.

\medskip{}

To define the map $\rho$ we consider the map
\[
\theta:S_{0}(p)_{m}^{\ord}\to S_{0}(p)_{et}^{\ord}
\]
defined on the moduli problems as
\[
\theta(\underline{A},H)=(\underline{A}^{(p^{2})},\Fr(\Ver^{-1}(H))).
\]
Then we let
\[
\rho=\theta\circ\sigma_{m}.
\]
Some words of explanation are in order. Here $\Ver^{-1}(H)$ is the
kernel of the isogeny $A^{(p)}\overset{\Ver}{\to}A\to A/H,$ hence
is finite flat of rank $p^{3m+n}$. For $H=A[p]^{mult},$ where $\Ver$
is an isomorphism from $H^{(p)}$ to $H,$ $\Ver^{-1}(H)$ coincides,
as a group functor, with $H^{(p)}+A^{(p)}[\Ver],$ so is seen to be
contained in $A^{(p)}[p].$ Also in this case, the kernel of $\Fr:A^{(p)}[p]\to A^{(p^{2})}[p]$
is contained in $\Ver^{-1}(H),$ hence the image $\Fr(\Ver^{-1}(H))$
is a finite flat subgroup scheme of rank $p^{2m}$. It is easily seen
that this subgroup scheme is $\mathcal{O}_{E}$-stable, Raynaud, isotropic
and \'etale, since these properties can be checked on the geometric
fibers. We also remark that for $H=A[p]^{mult}$
\[
\Fr(\Ver^{-1}(H))=\Fr(A^{(p)}[\Ver])=\Fr^{2}(A[p]).
\]
The reason we chose to define $\theta$ the way we did is that this
is the definition that will generalize later on, in \S\ref{subsec:The-maps-from S_0(p)_m and S_0(p)_et},
when $H$ is no longer multiplicative, to a map between the Zariski
closures $S_{0}(p)_{m}$ and $S_{0}(p)_{et}$ in $S_{0}(p)$.

\medskip{}

There is a similar map
\[
\theta':S_{0}(p)_{et}^{\ord}\to S_{0}(p)_{m}^{\ord}
\]
defined on the moduli problems as
\[
\theta'(\underline{A},H)=(\underline{A},\Ver(\Fr^{-1}(H^{(p^{2})}))).
\]
The proof that it is well-defined is similar to the one for $\theta.$
Note that
\[
\pi_{et}\circ\theta=Fr_{p}^{2}\circ\pi_{m},\,\,\,\,\pi_{m}\circ\theta'=\pi_{et},\,\,\,\,\theta\circ\theta'=Fr_{p}^{2},\,\,\,\,\theta'\circ\theta=Fr_{p}^{2}.
\]

\medskip{}

The definition of $\rho'$ is a little more subtle\footnote{We remind the reader that for general finite flat subgroup schemes
$\Gamma_{1}$ and $\Gamma_{2}$ of a finite flat group scheme $\Gamma,$
over an arbitrary locally Noetherian base, the subgroup scheme $\Gamma_{1}\cap\Gamma_{2}$
need not be flat, and the subgroup functor $\Gamma_{1}+\Gamma_{2}$
need not be represented by a group scheme at all. Similarly, for a
homomorphism $f$ between two finite flat group schemes, $\ker(f)$
is a subgroup scheme which need not be flat, and the group-functor
image of $f$ need not be represented by a group scheme at all.}. Let $(\underline{A}_{1},H_{1})\in(S_{0}(p)_{et}^{\ord})^{(p)}(R)$.
Thanks to the polarization, the subgroup scheme $H_{1}^{\perp}$ which
is the annihilator of $H_{1}$ relative to the Weil pairing on $A_{1}[p],$
is well-defined, and is finite flat of rank $p^{2n}$ over $R.$ We
claim that the closed subgroup scheme $H_{1}^{\perp}[\Fr]$ is finite
flat of rank $p^{n-m}.$ Indeed, it is enough to check it for the
universal $(\underline{A}_{1},H_{1})$ over $(S_{0}(p)_{et}^{\ord})^{(p)}.$
But now the base is reduced (by \cite{=00005BG=0000F6=00005D}) so
it is enough to check that all the geometric fibers of $H_{1}^{\perp}[\Fr]$
are of the same rank, and that this rank is $p^{n-m}.$ This is straightforward,
given that over an algebraically closed field we have the description
(\ref{eq:mu_ord_p_div_gp}). For any geometric point $x:\Spec(k)\to\Spec(R)$,
$H_{1}^{\perp}[\Fr]$ becomes the $\alpha_{p}$-subgroup which is
the kernel of Frobenius in the local-local part of $x^{*}(A_{1}[p]).$
Seen in another light, giving $H_{1}$ not only splits the connected-\'etale
exact sequence over $\Spec(R)$, but allows us to split off the local-local
part from the multiplicative part in $A_{1}[\Fr].$ In particular,
$H_{1}^{\perp}[\Fr]$ does not intersect $H_{1}$, so
\[
K_{1}:=H_{1}^{\perp}[\Fr]+H_{1}\simeq H_{1}^{\perp}[\Fr]\times H_{1}
\]
is finite flat of rank $p^{n+m}.$ This $K_{1}$ is a maximal isotropic
subgroup of $A_{1}[p],$ stable under $\mathcal{O}_E$, whose tangent space is $n-m$ dimensional,
of type $\overline{\Sigma}$ (sic!). Descending the polarization and
the endomorphisms to $A=A_{1}/K_{1}$ we get a principally polarized
abelian scheme over $R$, of type $(n,m).$ We let
\begin{equation}
\underline{A}=\rho'(\underline{A}_{1},H_{1}):=\left\langle p\right\rangle _{N}^{-1}\underline{A}_{1}/K_{1}\in S^{\ord}(R).\label{eq:Definition of rho'}
\end{equation}
The underlying principally polarized abelian scheme with endomorphisms
by $\mathcal{O}_{E}$ is $A.$ The level-$N$ structure differs from
the one descended from $\underline{A}_{1}$ by the diamond operator
$\left\langle p\right\rangle _{N}^{-1}.$ Recall that the \emph{diamond
operator} $\left\langle a\right\rangle _{N}$, for $a\in(\mathcal{O}_{E}/N\mathcal{O}_{E})^{\times}$
takes an $\mathcal{O}_{E}$-level-$N$ structure $\eta:(\mathcal{O}_{E}/N\mathcal{O}_{E})^{n+m}\simeq A[N]$
to $\eta\circ[a],$ where $[a]$ is multiplication by $a.$

\medskip{}

Having defined the maps in the diagram, we now check its commutativity.
We only have to check the commutativity of the top two triangles,
the bottom two being obvious. Let $\underline{A}_{1}$ and $\underline{A}$
be related by (\ref{eq:Definition of rho'}). Consider the morphism
$\Fr:\underline{A}\to\underline{A}^{(p)}$ obtained by dividing $A$
by $A[\Fr]$ and descending the polarization, the endomorphisms, and
the level structure. Since $A[\Fr]=A_{1}[p]/K_{1}$ we get the string
of isomorphisms
\begin{equation}
\underline{A}^{(p)}\simeq\underline{A}/A[\Fr]\simeq\left(\left\langle p\right\rangle _{N}^{-1}\underline{A}_{1}/K_{1}\right)/\left(A_{1}[p]/K_{1}\right)\simeq\left\langle p\right\rangle _{N}^{-1}\underline{A}_{1}/A_{1}[p]\simeq\underline{A}_{1}.\label{eq:string}
\end{equation}
(The last isomorphism is multiplication by $p,$ and it is the reason
for introducing $\left\langle p\right\rangle _{N}^{-1}$ in the definition
of $\rho'$.) We conclude that to accommodate an \'etale subgroup scheme
like $H_{1},$ $\underline{A}_{1}$ \emph{must be} of the form $\underline{A}^{(p)}$
and $K_{1}=A^{(p)}[\Ver].$ This \emph{may not be said} of $H_{1}$
itself, in general. However, if this is the case and $(\underline{A}_{1},H_{1})=(\underline{A}^{(p)},H^{(p)})=Fr_{p}((\underline{A},H))$
then the above discussion shows that $\rho'(\underline{A}_{1},H_{1})=\underline{A},$
proving the commutativity of the first triangle in the diagram:
\[
\rho'\circ Fr_{p}=\pi_{et}.
\]

For the second triangle consider
\[
\rho\circ\rho'(\underline{A}_{1},H_{1})=\rho(\underline{A})=(\underline{A}^{(p^{2})},\Fr(\Ver^{-1}(A[p]^{mult})))=(\underline{A}^{(p^{2})},\Fr^{2}(A[p]))
\]
\[
=(\underline{A}_{1}^{(p)},H_{1}^{(p)})=Fr_{p}((\underline{A}_{1},H_{1})).
\]
To justify the transition from the first to the second line, note
that $H_{1}^{(p)}=\Fr(H_{1})\subset\Fr^{2}(A[p])$ as $H_{1}\subset\Fr(A[p]),$
but $H_{1}^{(p)}$ and $\Fr^{2}(A[p])$ are both finite flat of rank
$p^{2m},$ so they coincide.
\begin{lem}
All the morphisms in the diagram are finite and flat.
\end{lem}

\begin{proof}
Since $Fr_{p}$ is a finite morphism between schemes of finite type
over a field, all the maps are clearly finite. The schemes are all
regular of dimension $nm.$ Finite maps between locally noetherian
regular schemes of the same dimension are flat (it is in fact enough
to assume that the source is Cohen-Macaulay). Note that by a theorem
of Kunz \cite{=00005BKu=00005D} the relative Frobenius morphism from
$X$ to $X^{(p)}$ is flat \emph{if and only if} $X$ is regular.
\end{proof}
\begin{lem}
The degrees of $\rho$ and $\rho'$ are given by 
\[
\deg(\rho)=p^{m^{2}},\,\,\,\,\,\deg(\rho')=p^{(n-m)m}.
\]
Furthermore, $\deg(\theta)=\deg(\rho)$ and $\deg(\theta')=\deg(\pi_{et})=p^{(2n-m)m}.$
\end{lem}

\begin{proof}
Since $\rho\circ\rho'=Fr_{p}$ is of degree $p^{nm}$, it is enough
to prove the formula for $\deg(\rho')$. Since $\rho'\circ Fr_{p}=\pi_{et}$
it is enough to prove that $\deg(\pi_{et})=p^{2nm-m^{2}}.$ We use
a method of degeneration from characteristic 0, based on the flatness
of $\pi:S_{0}(p)^{\ord}\to S^{\ord}.$ This map is flat because it
is finite, $S^{\ord}$ is regular and $S_{0}(p)^{\ord}$ is Cohen-Macaulay.
(In fact, the arguments of Proposition \ref{prop:S_0(p)^ord_m,et are smooth}
prove that $S_{0}(p)^{\ord}$ is non-singular.)

Denote by $\mathcal{S}^{\ord}$ the complement in $\mathcal{S}$ of
the non-ordinary locus in the special fiber, and similarly $\mathcal{S}_{0}(p)^{\ord}.$
Since $\mathcal{S}_{0}(p)^{\ord}$ and $\mathcal{S}^{\ord}$ are flat
over $\mathcal{O}_{E,(p)}$ and both the generic and special fibers
of $\pi:\mathcal{S}_{0}(p)^{\ord}\to\mathcal{S^{\ord}}$ are flat,
then by the criterion for flatness fiber-by-fiber $\pi$ is flat also
on the ordinary parts of the arithmetic schemes.

Fix a $W(k)$-valued point
\[
\xi:\Spec(W(k))\to\mathcal{S}^{\ord},
\]
and denote its specialization by $\xi_{0}:\Spec(k)\to S^{\ord}.$
Consider the pull-back
\[
\xi^{*}\mathcal{S}_{0}(p)\to\Spec(W(k)).
\]
As the base change of the finite flat morphism $\mathcal{S}_{0}(p)^{\ord}\to\mathcal{S^{\ord}}$,
this map is also finite flat. We denote by $\eta:\xi^{*}\mathcal{S}_{0}(p)\to\mathcal{S}_{0}(p)$
the base change of $\xi$ and by $\eta_{0}$ that of $\xi_{0}$.

\[\xymatrix{ \xi_0^\ast \mathcal{S}_0(p) \ar[r] \ar[d] \ar@/^1.5pc/[rr]^{\eta_0} & \xi^\ast \mathcal{S}_0(p) \ar[r]^\eta \ar[d] & \mathcal{S}_0(p) \ar[d] \\ \text{Spec}(k)\ar[r] \ar@/_1.5pc/[rr]_{\xi_0}& \text{Spec}(W(k)) \ar[r]^\xi & \mathcal{S} }\] 

Let $\xi^{*}\mathcal{S}_{0}(p)_{et}$ be the connected component of
$\xi^{*}\mathcal{S}_{0}(p)$ containing, in the special fiber, $\eta_{0}^{*}(S_{0}(p)_{et}^{\ord}).$
It is finite flat over $W(k),$ and the degree of its special fiber
is $\deg(\pi_{et}).$ We compute this degree in the \emph{generic}
fiber. Let $A=\xi^{*}\mathcal{A}$ be the pull-back of the universal
abelian scheme to $W(k),$ and let $A_{0}$ be its special fiber.
Consider $A(\overline{E}_{p})[p]$, the $p$-torsion in the group
of points of $A$ in a fixed algebraic closure $\overline{E}_{p}$
of the local field $E_{p}$. The geometric points in the generic fiber
of $\xi^{*}\mathcal{S}_{0}(p)_{et}$ are in 1-1 correspondence with
the isotropic, $\mathcal{O}_{E}$-stable subgroups $H\subset A(\overline{E}_{p})[p]$
of rank $p^{2m}$ specializing to $A_{0}[p]^{et}\subset A_{0}[p].$

Denote by $A(\overline{E}_{p})^{0}$ the kernel of the reduction map
$A(\overline{E}_{p})\to A_{0}(k).$ We have to count isotropic, $\mathcal{O}_{E}$-stable
subgroups $H\subset A(\overline{E}_{p})[p]$ of rank $p^{2m}$ satisfying
\[
H+A(\overline{E}_{p})^{0}[p]=A(\overline{E}_{p})[p].
\]

We are now reduced to linear algebra. The $\mathcal{O}_{E}$-module
$A(\overline{E}_{p})[p]$ with the hermitian pairing derived from
the polarization is isomorphic to $\kappa^{n+m}$ with the $\kappa$-hermitian
form
\[
(u,v)=\,^{t}u^{(p)}\left(\begin{array}{ccc}
 &  & 1_{m}\\
 & 1_{n-m}\\
1_{m}
\end{array}\right)v
\]
and we may choose the isomorphism so that $A(\overline{E}_{p})^{0}[p]$
is the subspace with the last $m$ entries $0.$ We thus have to count
equivalence classes of $\kappa$-linear maps $\lambda:\kappa^{m}\hookrightarrow\kappa^{n+m}$
satisfying (1) the image of $\lambda$ is isotropic, and (2) the projection
of the image of $\lambda$ on the last $m$ coordinates is an isomorphism.
Two such maps $\lambda_{1}$ and $\lambda_{2}$ are equivalent if
$\lambda_{2}=\lambda_{1}\circ\alpha$ for $\alpha\in GL_{m}(\kappa).$
This is the same as counting matrices $\Gamma\in M_{n\times m}(\kappa)$
satisfying
\[
(\,^{t}\Gamma^{(p)},1_{m})\left(\begin{array}{ccc}
 &  & 1_{m}\\
 & 1_{n-m}\\
1_{m}
\end{array}\right)\left(\begin{array}{c}
\Gamma\\
1_{m}
\end{array}\right)=0,
\]
or equivalently, counting pairs $(\Gamma_{1},\Gamma_{2})\in M_{m\times m}(\kappa)\times M_{(n-m)\times m}(\kappa)$
satisfying
\[
\Gamma_{1}+\,^{t}\Gamma_{1}^{(p)}+\,^{t}\Gamma_{2}^{(p)}\Gamma_{2}=0.
\]
This number is easily seen to be $p^{2nm-m^{2}}.$ Indeed, there are
$p^{2(n-m)m}$ choices for $\Gamma_{2}.$ For each choice of $\Gamma_{2}$
there are $p^{2\frac{m(m-1)}{2}}$ choices for the entries of $\Gamma_{1}$
above the diagonal, which are arbitrary and determine the entries
below the diagonal uniquely, and $p^{m}$ choices for the entries
on the diagonal. This concludes the proof of the lemma.
\end{proof}

\subsubsection{The relation between the foliation and $\rho$}

In \S\ref{subsec:Definition-and-first foliation} we have constructed
the foliation $\mathcal{T}S^{+}$ in the tangent bundle of $S^{\ord},$
while in \S\ref{subsec:Morphisms-between-ordinary-=0000E9tale} we
have constructed a flat height 1 morphism $\rho:S^{\ord}\to S_{0}(p)_{et}^{\ord}$.
We shall now prove that the two correspond to each other under the
dictionary between height 1 morphisms and height 1 foliations discussed
in Proposition \ref{prop:=00005BEk=00005D,-Proposition-2.4.}.
\begin{thm}
\label{thm:fol-mor corresp}The quotient of $S^{\ord}$ by the height
1 foliation $\mathcal{T}S^{+}$ is the height 1 morphism $\rho:S^{\ord}\to S_{0}(p)_{et}^{\ord}.$
\end{thm}

\begin{proof}
We have to show that the image of $\rho^{*}(\Omega_{S_{0}(p)_{et}^{\ord}})$
in $\Omega_{S^{\ord}}$ is $KS(\mathcal{P}_{0}\otimes\mathcal{Q})$.
Since $\rho$ is a finite flat height 1 morphism of degree $p^{m^{2}}$
we know, from the general theory explained in \S\ref{subsec:Foliations-and-inseparable},
that the image of $\rho^{*}(\Omega_{S_{0}(p)_{et}^{\ord}})$ in $\Omega_{S}$
is a sub-bundle of $\Omega_{S^{\ord}}$ of rank $(n-m)m.$ Since the
same is true of $KS(\mathcal{P}_{0}\otimes\mathcal{Q})$, it is enough
to prove the inclusion
\[
KS(\mathcal{P}_{0}\otimes\mathcal{Q})\subset\mathrm{Im}(\rho^{*}(\Omega_{S_{0}(p)_{et}^{\ord}})\to\Omega_{S^{\ord}}).
\]
As the right hand side is equal to $\ker(\rho'^{*}:\Omega_{S^{\ord}}\to\Omega_{(S_{0}(p)_{et}^{\ord})^{(p)}})$
(they are both sub-bundles of rank $(n-m)m$ and the image of $\rho^{*}$
is contained in the kernel of $\rho'^{*}$ because $\rho\circ\rho'=Fr_{p}),$
it will be enough to prove that $KS(\mathcal{P}_{0}\otimes\mathcal{Q})$
is contained in the latter. More precisely, we have to show that
\[
\rho'^{*}KS(\mathcal{P}_{0}\otimes\mathcal{Q})\subset\ker(\rho'^{*}:\rho'^{*}\Omega_{S^{\ord}}\to\Omega_{(S_{0}(p)_{et}^{\ord})^{(p)}}).
\]

For this purpose consider the universal pair $(\mathcal{A}_{1},\mathcal{H}_{1})$
over $(S_{0}(p)_{et}^{\text{ord}})^{(p)}$. In (\ref{eq:string})
we have constructed $\mathcal{A}$ such that, at the level of points,
$\rho^{\prime}((\mathcal{A}_{1},\mathcal{H}_{1}))=\mathcal{A}$ (where
to simplify typesetting, we omit the underline symbol). Note that
$\mathcal{A}$ is a scheme over $S^{\text{ord}}$ and, letting $\mathcal{B}=\rho^{\prime\ast}\mathcal{A}$,
the relations obtained in (\ref{eq:string}) imply a canonical isomorphism
$\mathcal{A}_{1}=\mathcal{B}^{(p)}$, as abelian schemes over $(S_{0}(p)_{et}^{\text{ord}})^{(p)}$.
The construction also provides a canonical isogeny $\mathcal{A}_{1}\rightarrow\mathcal{B}$,
which is nothing else than $\text{Ver}$. 

The kernel of $\Ver$ contains the finite flat group scheme $\mathcal{H}_{1}$.
Thus, letting $\mathcal{C}=\mathcal{A}_{1}/\mathcal{H}_{1}$, we get
a decomposition of $\Ver$ as the composition of two isogenies between
abelian schemes over $(S_{0}(p)_{et}^{\text{ord}})^{(p)}$:

\[ \xymatrix{\mathcal A_1 \ar[r]^\psi\ar@/_1pc/[rr]_{\text{Ver}} & \mathcal C \ar[r]^\varphi& \mathcal B}.\]Here
$\psi$ is the isogeny with kernel $\mathcal{H}_{1}$, and $\varphi$
is the isogeny with kernel $\mathcal{A}_{1}[\text{Ver}]/\mathcal{H}_{1}$.
Note that although $\text{Ver}:\mathcal{A}_{1}\rightarrow\mathcal{B}$
is a pull-back by $\rho^{\prime}$ of a similar isogeny over $S^{\text{ord}}$,
only over $(S_{0}(p)_{et}^{\text{ord}})^{(p)}$ does it factor through
$\mathcal{C}$, because $\mathcal{H}_{1}$ is \emph{not} the pull-back
of a group scheme on $S^{\text{ord}}$.

Consider now the commutative diagram
\begin{equation}
\mathnormal{\begin{array}{ccc}
\rho'^{*}\mathcal{P}=\omega_{\mathcal{B}}(\Sigma) & \overset{KS_{\mathcal{B}}}{\longrightarrow} & \Omega_{(S_{0}(p)_{et}^{\ord})^{(p)}}\otimes R^{1}\pi_{*}\mathcal{O}_{\mathcal{B}}(\Sigma)\\
\downarrow\varphi^{*} &  & \downarrow\textsc{1\ensuremath{\otimes\varphi^{*}}}\\
\omega_{\mathcal{C}}(\Sigma) & \overset{KS_{\mathcal{C}}}{\longrightarrow} & \Omega_{(S_{0}(p)_{et}^{\ord})^{(p)}}\otimes R^{1}\pi_{*}\mathcal{O}_{\mathcal{C}}(\Sigma)
\end{array}}\label{eq:KS_diagram}
\end{equation}
resulting from the functoriality of the Gauss-Manin connection with
respect to the isogeny $\varphi.$ Here $KS_{\mathcal{B}}$ is the
Kodaira-Spencer map for the family $\mathcal{B}\to(S_{0}(p)_{et}^{\ord})^{(p)}$
and likewise for $\mathcal{C}$.

The kernel of the left vertical arrow $\varphi^{*}$ is precisely
$\rho'^{*}(\mathcal{P}_{0}).$ This is because $\psi$ is \'etale, so
$\psi^{*}$ is an isomorphism on cotangent spaces, hence
\[
\ker(\varphi^{*}|_{\rho'^{*}\mathcal{P}})=\ker(\Ver^{*}|_{\rho'^{*}\mathcal{P}})=\rho'^{*}(\mathcal{P}_{0}).
\]

On the right side of (\ref{eq:KS_diagram}), we claim that $1\otimes\varphi^{*}$
is injective. To verify it, consider the commutative diagram
\begin{equation}
\begin{array}{ccccccccc}
0 & \to & \omega_{\mathcal{B}}(\Sigma) & \to & H_{dR}^{1}(\mathcal{B}/S)(\Sigma) & \to & R^{1}\pi_{*}\mathcal{O}_{\mathcal{B}}(\Sigma) & \to & 0\\
 &  & \downarrow\varphi^{*} &  & \downarrow\varphi^{*} &  & \downarrow\varphi^{*}\\
0 & \to & \omega_{\mathcal{C}}(\Sigma) & \to & H_{dR}^{1}(\mathcal{C}/S)(\Sigma) & \to & R^{1}\pi_{*}\mathcal{O}_{\mathcal{C}}(\Sigma) & \to & 0
\end{array}.\label{eq:BCisogeny}
\end{equation}
The right vertical arrow may be identified with the $\Sigma$-component
of the map $\varphi_{*}^{t}:\Lie(\mathcal{B}^{t})\to\Lie(\mathcal{C}^{t})$.
The signature of $\mathcal{B}^{t}$ is $(m,n)$, and at every geometric
point $x$
\[
\mathcal{B}_{x}^{t}[p]\simeq(\mathcal{O}_{E}\otimes\mu_{p})^{m}\oplus(\mathscr{G}_{\overline{\Sigma}}[p])^{n-m}\oplus(\mathcal{O}_{E}\otimes\mathbb{Z}/p\mathbb{Z})^{m}.
\]
As $\ker(\varphi)=\mathcal{A}_{1}[\text{Ver}]/\mathcal{H}_{1}$ is
local-local, so is its dual $\ker(\varphi^{t})$; in fact, its geometric
fibers are all isomorphic to $\alpha_{p}^{n-m}\subset(\mathscr{G}_{\overline{\Sigma}}[p])^{n-m}$\emph{.
}The Lie algebra of $\ker(\varphi^{t}),$ i.e. $\ker(\varphi_{*}^{t})$,\emph{
is therefore of type} $\overline{\Sigma}$. We conclude that the right
vertical arrow of (\ref{eq:BCisogeny}), and with it the right vertical
arrow of (\ref{eq:KS_diagram}), are injective, as claimed.

As $\mathcal{B}=\rho'^{*}\mathcal{A}$, the morphism $KS_{\mathcal{B}}$
is the composition of the map
\[
\rho'^{*}(KS_{\mathcal{A}}):\rho'^{*}(\mathcal{P})\to\rho'^{*}(\Omega_{S^{\ord}})\otimes\rho'^{*}(\mathcal{Q}{}^{\vee})
\]
(we identify $\mathcal{Q=\omega_{\mathcal{A}}}(\overline{\Sigma})$
with $\omega_{\mathcal{A}^{t}}(\Sigma)$ via the polarization as usual,
hence also $R^{1}\pi_{*}\mathcal{O}_{\mathcal{A}}(\Sigma)=\Lie(\mathcal{A}^{t})(\Sigma)$
with $\mathcal{Q}^{\vee}$) and the map induced by
\[
\rho'^{*}:\rho'^{*}(\Omega_{S^{\ord}})\to\Omega_{(S_{0}(p)_{et}^{\ord})^{(p)}}.
\]

From the commutativity of (\ref{eq:KS_diagram}) we conclude that
$KS_{\mathcal{B}}(\rho'^{*}(\mathcal{P}_{0}))=0,$ hence the desired
inclusion
\[
\rho'^{*}(KS(\mathcal{P}_{0}\otimes\mathcal{Q}))\subset\ker(\rho'^{*}:\rho'^{*}\Omega_{S^{\ord}}\to\Omega_{(S_{0}(p)_{et}^{\ord})^{(p)}}).
\]
 
\end{proof}

\subsection{Moonen's generalized Serre-Tate coordinates}

Although not necessary for the rest of the paper, we digress to explain
the relation between $\mathcal{T}S^{+}$ and the generalized Serre-Tate
coordinates introduced by Moonen. For the following proposition see
\cite{=00005BMo1=00005D}, the remark at the end of Example 3.3.2,
and 3.3.3(d) (case AU, $r=3$).
\begin{prop}
\label{Moonen}Let $x\in S^{\ord}(k).$ Let $\mathfrak{\widehat{\mathscr{G}}}$
be the formal group over $k$ associated with the $p$-divisible group
$\mathfrak{\mathscr{G}}$ and let $\mathbb{\widehat{G}}_{m}$ be the
formal multiplicative group over $k.$ Then the formal neighborhood
$Spf(\widehat{\mathcal{O}}_{S,x})$ of $x$ has a natural structure
of a $\mathbb{\widehat{G}}_{m}^{m^{2}}$-torsor over $\mathfrak{\widehat{\mathscr{G}}}^{(n-m)m}$.
This torsor is obtained as the set of symmetric elements under the
involution induced by the polarization on a certain bi-extension of
$\mathfrak{\widehat{\mathscr{G}}}^{(n-m)m}\times\widehat{\mathscr{G}}^{(n-m)m}$
, hence it contains a canonical formal torus $\widehat{T}_{x}$ sitting
over the origin of $\mathfrak{\widehat{\mathscr{G}}}^{(n-m)m}$.
\end{prop}

\begin{thm}
\label{thm:compatibility with Moonen}Let $x\in S^{\ord}(k)$. Then
$\mathcal{T}S^{+}|_{x}$ is the tangent space to $\widehat{T}_{x}\subset Spf(\widehat{\mathcal{O}}_{S,x})$.
\end{thm}

\begin{proof}
Let $i:\widehat{T}_{x}\hookrightarrow Spf(\widehat{\mathcal{O}}_{S,x})$
be the embedding of formal schemes given by Proposition \ref{Moonen}.
It sends the origin $e$ of $\widehat{T}_{x}$ to $x.$ Let $i_{*}$
be the induced map on tangent spaces
\[
i_{*}:\mathcal{T}\widehat{T}_{x}|_{e}\hookrightarrow\mathcal{T}S|_{x}.
\]
We have to show that $i_{*}(\mathcal{T}\widehat{T}_{x}|_{e})$ annihilates
$KS(\mathcal{P}_{0}\otimes\mathcal{Q})|_{x}.$ This is equivalent
to saying that when we consider the pull back $i^{*}\mathcal{A}$
of the universal abelian scheme to $\widehat{T}_{x}$, \emph{its}
Kodaira-Spencer map kills $\mathcal{P}_{0}\otimes\mathcal{Q}|_{e}$.
For this recall the definition of $KS=KS(\Sigma)$.

Let $\mathfrak{S}=\widehat{T}_{x}$ and write for simplicity $\mathcal{A}$
for $i^{*}\mathcal{A}$. We then have the following commutative diagram
\begin{equation}
\begin{array}{ccc}
\mathcal{P}=\omega_{\mathcal{A}/\mathfrak{S}}(\Sigma) & \hookrightarrow & H_{dR}^{1}(\mathcal{A}/\mathfrak{S})(\Sigma)\\
\downarrow KS &  & \downarrow\nabla\\
\mathcal{Q}^{\vee}\otimes\Omega_{\mathfrak{S}}^{1}\simeq\omega_{\mathcal{A}^{t}/\mathfrak{S}}^{\vee}(\Sigma)\otimes\Omega_{\mathfrak{S}}^{1} & \longleftarrow & H_{dR}^{1}(\mathcal{A}/\mathfrak{S})(\Sigma)\otimes\Omega_{\mathfrak{S}}^{1}
\end{array}\label{eq:KS}
\end{equation}
in which we identified $H^{1}(\mathcal{A},\mathcal{O})$ with $H^{0}(\mathcal{A}^{t},\Omega_{\mathcal{A}^{t}/\mathfrak{S}}^{1})^{\vee}$
and used the polarization to identify the latter with $\omega_{\mathcal{A}/\mathfrak{S}}^{\vee}$,
reversing types. Here $\nabla$ is the Gauss-Manin connection, and
the tensor product is over $\mathcal{\widehat{O}_{\mathfrak{S}}}.$
Although $\nabla$ is a derivation, $KS$ is a homomorphism of vector
bundles over $\mathcal{\widehat{O}_{\mathfrak{S}}}$. We shall show
that $KS(\mathcal{P}_{0})=0,$ where $\mathcal{P}_{0}=\ker(V:\omega_{\mathcal{A}/\mathfrak{S}}\to\omega_{\mathcal{A}/\mathfrak{S}}^{(p)})\cap\mathcal{P}$.

At this point recall the filtration
\[
0\subset Fil^{2}=\mathcal{A}[p^{\infty}]^{mult}\subset Fil^{1}=\mathcal{A}[p^{\infty}]^{conn}\subset Fil^{0}=\mathcal{A}[p^{\infty}]
\]
of the $p$-divisible group of $\mathcal{A}$ over $\mathfrak{S}.$
The graded pieces are of height $2m,$ $2(n-m)$ and $2m$ respectively,
and $\mathcal{O}_{E}$-stable. They are given by
\[
gr^{2}=(\mathcal{O}_{E}\otimes\mu_{p^{\infty}})^{m},\,\,\,gr^{1}=\mathscr{G}_{\Sigma}^{n-m},\,\,\,gr^{0}=(\mathcal{\mathcal{O}}_{E}\otimes\mathbb{Q}_{p}/\mathbb{Z}_{p})^{m}.
\]

For any $p$-divisible group $G$ over $\mathfrak{S}$ denote by $\mathbb{D}(G)$
the Dieudonn\'e crystal associated to $G$, and let $D(G)=\mathbb{D}(G)_{\mathfrak{S}},$
\emph{cf. }\cite{=00005BGro=00005D}. The $\mathcal{\widehat{O}_{\mathfrak{S}}}$-module
$D(G)$ is endowed with an integrable connection $\nabla$ and the
pair $(D(G),\nabla)$ determines $\mathbb{D}(G).$

In our case, we can identify $D(\mathcal{A}[p^{\infty}])$ with $H_{dR}^{1}(\mathcal{A}/\mathfrak{S})$,
and the connection with the Gauss-Manin connection. The above filtration
on $\mathcal{A}[p^{\infty}]$ induces therefore a filtration $Fil^{\bullet}$
on $H_{dR}^{1}(\mathcal{A}/\mathfrak{S})$ which is preserved by $\nabla$.
Since the functor $\mathbb{D}$ is contravariant, we write the filtration
as
\[
0\subset Fil^{1}H_{dR}^{1}(\mathcal{A}/\mathfrak{S})\subset Fil^{2}H_{dR}^{1}(\mathcal{A}/\mathfrak{S})\subset Fil^{3}=H_{dR}^{1}(\mathcal{A}/\mathfrak{S})
\]
where 
\[
Fil^{i}H_{dR}^{1}(\mathcal{A}/\mathfrak{S})=D(\mathcal{A}[p^{\infty}]/Fil^{i}\mathcal{A}[p^{\infty}]).
\]
For example, $Fil^{1}H_{dR}^{1}(\mathcal{A}/\mathfrak{S})$ is sometimes
referred to as the ``unit root subspace''. As $Fil^{2}\mathcal{A}[p^{\infty}]$
is of multiplicative type, $\ker(V:H_{dR}^{1}(\mathcal{A}/\mathfrak{S})\to H_{dR}^{1}(\mathcal{A}/\mathfrak{S})^{(p)})$
is contained in $Fil^{2}H_{dR}^{1}(\mathcal{A}/\mathfrak{S})$. In
particular,
\[
\mathcal{P}_{0}\subset Fil^{2}H_{dR}^{1}(\mathcal{A}/\mathfrak{S}).
\]
Let $G=\mathcal{A}[p^{\infty}]/\mathcal{A}[p^{\infty}]^{mult}$, so
that $Fil^{2}H_{dR}^{1}(\mathcal{A}/\mathfrak{S})=D(G).$ It follows
that in computing $KS$ on $\mathcal{P}_{0}$ we may use the following
diagram instead of (\ref{eq:KS}):

\begin{equation}
\begin{array}{ccc}
\mathcal{P}_{0} & \hookrightarrow & D(G)(\Sigma)\\
\downarrow KS &  & \downarrow\nabla\\
\mathcal{Q}^{\vee}\otimes\Omega_{\mathfrak{S}}^{1} & \longleftarrow & D(G)(\Sigma)\otimes\Omega_{\mathfrak{S}}^{1}
\end{array}\label{eq:KS2}
\end{equation}

Finally, we have to use the description of the formal neighborhood
of $x$ as given in \cite{=00005BMo1=00005D}. Since we are considering
the pull-back of $\mathcal{A}$ to $\mathfrak{S}$ only, and not the
full deformation over $Spf(\widehat{\mathcal{O}}_{S,x})$, it follows
from the construction of the 3-cascade (biextension) in \emph{loc.cit.
}\S2.3.6 that the $p$-divisible groups $Fil^{1}\mathcal{A}[p^{\infty}]$,
and dually $G=\mathcal{A}[p^{\infty}]/Fil^{2}$, are \emph{constant}
over $\mathfrak{S}$. Thus over $\mathfrak{S}$
\[
G\simeq\mathscr{G}^{n-m}\times(\mathcal{O}_{E}\otimes\mathbb{Q}_{p}/\mathbb{Z}_{p})^{m},
\]
and $\nabla$ maps $D(\mathscr{G}^{n-m})$ to $D(\mathscr{G}^{n-m})\otimes\Omega_{\mathfrak{S}}^{1}$.
Since
\[
\mathcal{P}_{0}=\omega_{\mathscr{G}^{n-m}}=D(\mathscr{G}^{n-m})(\Sigma)
\]
as subspaces of $H_{dR}^{1}(\mathcal{A}/\mathfrak{S}),$ 
\[
\nabla(\mathcal{P}_{0})\subset\mathcal{P}_{0}\otimes\Omega_{\mathfrak{S}}^{1}.
\]
The bottom arrow in (\ref{eq:KS2}) comes from the homomorphism
\[
D(G)(\Sigma)\hookrightarrow H_{dR}^{1}(\mathcal{A}/\mathfrak{S})(\Sigma)\overset{pr}{\to}H^{1}(\mathcal{A},\mathcal{O})(\Sigma)\overset{\phi}{\simeq}H^{1}(\mathcal{A}^{t},\mathcal{O})(\overline{\Sigma})=\mathcal{Q}^{\vee}.
\]
But the projection $pr$ kills $\mathcal{P}_{0}\subset\omega_{\mathcal{A}/\mathfrak{S}}.$
This concludes the proof.
\end{proof}
\textbf{Remark. }Proposition \ref{Moonen} yields a natural integral
``formal submanifold'' to the height 1 foliation $\mathcal{T}S^{+}$
in a formal neighborhood of any ordinary point. As mentioned in 
S\ref{subsec:Foliations-and-inseparable},
integral submanifolds to height 1 foliations are ubiquitous. On the
other hand we were not able to lift $\mathcal{T}S^{+}$ to an $h$-foliation
in the sense of \cite{=00005BEk=00005D} for $h>1,$ and we do not
believe that they lift to characteristic 0 as in \cite{=00005BMi=00005D}.
The meaning of these ``natural'' formal submanifolds from the point
of view of foliations remains mysterious.

\section{\label{sec:Extending-the-ordinary}Extending the foliation beyond
the ordinary locus}

In this section we discuss the extension of the foliation $\mathcal{T}S^{+}$
from $S^{\ord}$ to a certain ``successive blow-up'' $S^{\sharp}$
of $S.$ We define a finite flat morphism from $S^{\sharp}$ to the
Zariski closure $S_{0}(p)_{et}$ of $S_{0}(p)_{et}^{\ord}$, extending
the morphism from $S^{\ord}$ to $S_{0}(p)_{et}^{\ord}$, and show
that this map is the quotient by the extended foliation. In this section
we shall use the results on the Ekedahl-Oort stratification summarized
in \S\ref{subsec:The-NP-and EO}.

\subsection{\label{subsec:The-moduli-scheme S=000023}The moduli scheme $S^{\sharp}$}

\subsubsection{Definition and general properties}

Recall (\S\ref{subsec:Frobenius,-Verschiebung-and}) that the map $V_{\mathcal{P}}$
induced by Verschiebung maps $\mathcal{P}$ to $\mathcal{Q}^{(p)}$,
hence its kernel at any point of $S$ is at least $(n-m)$-dimensional.
Over $S^{\ord},$ but not only there, $V_{\mathcal{P}}$ maps $\mathcal{P}$
onto $\mathcal{Q}^{(p)},$ so the kernel $\mathcal{P}[V]$ is precisely
of dimension $n-m.$ 

Define a moduli problem $S^{\sharp}$ on $\kappa$-algebras $R$ by
setting
\[
S^{\sharp}(R)=\left\{ (\underline{A},\mathcal{P}_{0})|\,\,\underline{A}\in S(R),\,\,\mathcal{P}_{0}\subset\mathcal{P}[V]\,\,\text{a subbundle of rank \ensuremath{n-m}}\right\} .
\]
There is a forgetful map $f:S^{\sharp}\to S$, which is bijective
over $S^{\ord}.$ Let $Gr(n-m,\mathcal{P})$ be the relative Grassmanian
over $S$ classifying sub-bundles $\mathcal{N}$ of rank $n-m$ in
$\mathcal{P}.$ It is a smooth scheme over $S$, of relative dimension
$(n-m)m.$ As the condition $V(\mathcal{N})=0$ is closed, the moduli
problem $S^{\sharp}$ is representable by a closed subscheme of $Gr(n-m,\mathcal{P})$.
The fiber $S_{x}^{\sharp}=f^{-1}(x)$ is the Grassmanian of $(n-m)$-dimensional
subspaces in $\mathcal{P}_{x}[V]$, and if $x\in S_{w}$ its dimension
is, in the notation of \S\ref{subsec:The-NP-and EO},
\[
\dim S_{x}^{\sharp}=(n-m)(a_{\Sigma}(w)-n+m).
\]

Denote by $S_{\sharp}$ the open subset of $S$ where $f$ is an isomorphism,
i.e. where $a_{\Sigma}(w)=n-m.$ It is a union of EO strata, containing
$S^{\ord}.$ 

For any $(n,m)$-shuffle $w$ denote the pre-image of the EO stratum
$S_{w}$ by
\[
S_{w}^{\sharp}=f^{-1}(S_{w}).
\]
\begin{prop}
The open set $S_{\sharp}$ contains $\binom{n}{m}$ EO strata. It
contains a unique minimal stratum in the EO order, denoted $S_{\fol}$,
which is of dimension $m^{2}.$
\end{prop}

\begin{proof}
Using the labeling of the EO strata by the set $\Pi(n,m)$ of $(n,m)$-shuffles
in $\mathfrak{S}_{n+m}$, and formula (\ref{eq:sigma_part_a_number}),
we see that $a_{\Sigma}(w)=n-m$ if and only if
\[
w^{-1}(n-m+j)=n+j
\]
for all $1\le j\le m.$ Thus, the set of $w$ satisfying $a_{\Sigma}(w)=n-m$
is in bijection with $\Pi(n-m,m),$ the set of $(n-m,m)$-shuffles
in $\mathfrak{S}_{n}.$ More precisely, we have to arrange the numbers
\[
\{1,\dots,n-m;n+1,\dots n+m\}
\]
in the interval $[1,n]$, preserving the order within each block.
There are $\binom{n}{m}$ such shuffles.

Let $\iota:\Pi(n-m,m)\hookrightarrow\Pi(n,m)$ be the inclusion described
above. The element
\[
w_{\fol}=\left(\begin{array}{ccccccccc}
1 & \dots & n-m & n-m+1 & \dots & n & n+1 & \dots & n+m\\
1 & \dots & n-m & n+1 & \dots & n+m & n-m+1 & \dots & n
\end{array}\right)
\]
belongs to $\iota(\Pi(n-m,m))$ and is the unique minimal element
there in the usual Bruhat order. From the remark at the end of \S\ref{subsec:The-NP-and EO}
we deduce that it is also the unique minimal element among $\iota(\Pi(n-m,m))$
in the EO order $\preceq.$ This $w_{\fol}$ must represent a stratum
$S_{\fol}=S_{w_{\fol}}$ of minimal dimension among the EO strata
in $S_{\sharp}.$ Since $\dim S_{w}=l(w)$ we conclude that its dimension
is $l(w_{\fol})=m^{2}.$
\end{proof}
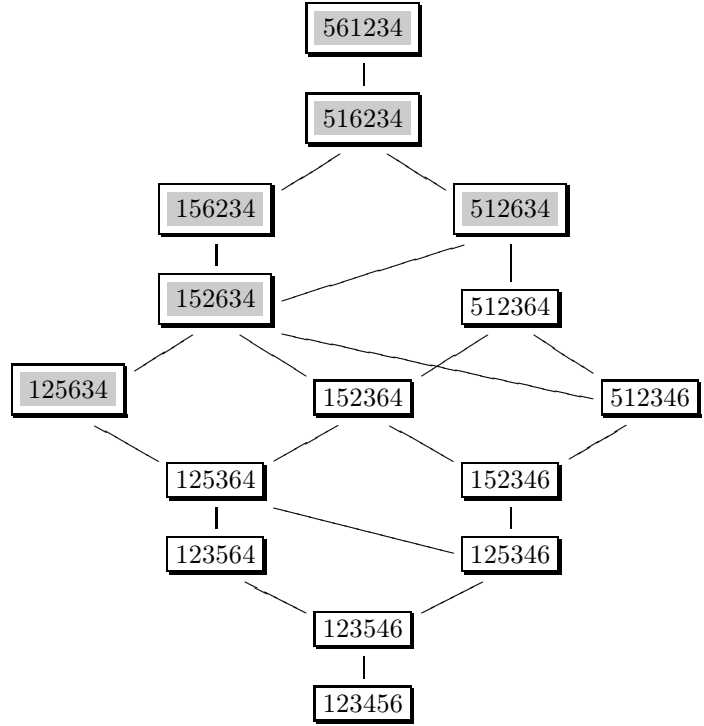
\begin{figure}[h]
\caption{\label{fig:The-EO-strata}The EO strata for $(n,m)=(4,2)$}

\definecolor{light-gray}{gray}{0.8}
\setlength{\shadowsize}{1pt}
\[\xymatrix@C=6pt@R=8pt{&&\shadowbox{\colorbox{light-gray}{561234}}\ar@{-}[d]&&\\
&&\shadowbox{\colorbox{light-gray}{516234}}\ar@{-}[dl]\ar@{-}[dr]&&\\
& \shadowbox{\colorbox{light-gray}{156234}}\ar@{-}[d] & &\shadowbox{\colorbox{light-gray}{512634}}\ar@{-}[dll]\ar@{-}[d]&\\
& \shadowbox{\colorbox{light-gray}{152634}}\ar@{-}[dl]\ar@{-}[dr]\ar@{-}[drrr] & &\shadowbox{512364}\ar@{-}[dl]\ar@{-}[dr]&\\
\shadowbox{\colorbox{light-gray}{125634}}\ar@{-}[dr]& & \shadowbox{152364}\ar@{-}[dl]\ar@{-}[dr]&&\shadowbox{512346}\ar@{-}[dl]\\
&\shadowbox{125364}\ar@{-}[d]\ar@{-}[drr] && \shadowbox{152346}\ar@{-}[d]&\\
&\shadowbox{123564}\ar@{-}[dr] && \shadowbox{125346}\ar@{-}[dl]&\\
& & \shadowbox{123546}\ar@{-}[d]&&\\
& & \shadowbox{123456}&&\\
}\]
\end{figure}

In Figure \ref{fig:The-EO-strata}, taken from \cite{=00005BWoo=00005D},
we illustrate the EO stratification when $(n,m)=(4,2)$. There are
15 EO strata altogether, labeled by $(4,2)$-shuffles $w$. We write
the $(4,2)$-shuffle $w$ as $w(1)\dots w(6).$ The strata are arranged
from top to bottom in rows, according to their dimension (equal to
the length of $w$). The top row contains only $S^{\ord},$ whose
dimension is 8, and the bottom row contains only the core stratum
in dimension 0. The EO order relation is represented by downward lines.
The 6 strata in $S_{\sharp}$ are those in which $w$ ends with $(...34)$.
The lowest one, $S_{\fol}$, has dimension 4. Note that $S_{\sharp}$
contains two 6-dimensional strata.

\medskip{}

By construction, $S^{\sharp}$ carries a tautological sub-bundle
\[
\mathcal{P}_{0}\subset f^{*}\mathcal{P}
\]
of rank $n-m,$ which extends the sub-bundle $\mathcal{P}_{0}$ defined
on $S^{\ord}.$ As long as we are above $S_{\sharp}$ it can be viewed
as a bundle on $S.$

\subsubsection{Example: the case of U(n,1)}

This case is particularly simple. There are $n+1$ EO strata, and
all of them, except for the core points, lie in $S_{\sharp}.$ The
fibers of $S^{\sharp}$ at the core points are projective spaces of
dimension $n-1$. In fact, at such a point $x\in S$ we have $\mathcal{P}_{x}[V]=\mathcal{P}_{x},$
because $\mathcal{A}_{x}$ is superspecial, so there is a \emph{canonical
}identification
\[
S_{x}^{\sharp}=Gr(n-1,\mathcal{P}_{x})=Gr(n-1,\mathcal{P}_{x}\otimes\mathcal{Q}_{x})
\]
because $\mathcal{Q}$ is a line bundle. But under the Kodaira-Spencer
map $\mathcal{P}_{x}\otimes\mathcal{Q}_{x}$ is identified with the
cotangent space of $S$ at $x,$ so by duality we have a canonical
identification of $S_{x}^{\sharp}$ with $Gr(1,\mathcal{T}S_{x}).$
In fact, $S^{\sharp}$ is the blow-up of $S$ at the core points.

\subsubsection{Smoothness and irreducibility}
\begin{thm}
\label{thm:smoothness of S=000023}The scheme $S^{\sharp}$ is non-singular
and $f$ induces a bijection on irreducible components.
\end{thm}

\begin{proof}
We work over an algebraically closed field $k$ containing $\kappa$.
Let $y\in S^{\sharp}(k)$ and $x=f(y).$ Let $k[\epsilon]$ be the
ring of dual numbers. Denote by $S^{\sharp}(k[\epsilon])_{y}$ the
tangent space at $y$ to the moduli problem $S^{\sharp}.$ This is
the set of elements in $S^{\sharp}(k[\epsilon])$ mapping to $y$
modulo $\epsilon,$ equipped with the natural structure of a $k$-vector
space. We shall show that 
\begin{equation}
\dim S^{\sharp}(k[\epsilon])_{y}=nm.\label{eq:dimension}
\end{equation}

Let us first see how this implies the theorem. Let $Y$ be the Zariski
closure of an irreducible component of $S^{\ord}$ in $S^{\sharp}.$
It is $nm$ dimensional, hence (\ref{eq:dimension}), applied to $y\in Y$,
shows that $Y$ is non-singular, and any other irreducible component
of $S^{\sharp}$ is disjoint from $Y.$ Since the fibers of $f$ are
connected, there do not exist any other irreducible components in
$f^{-1}(f(Y))$. It remains to prove (\ref{eq:dimension}).

Standard techniques in deformation theory show that we have to compute
the tangent space to a certain incidence variety between Grassmanians
(see \cite{=00005BHar=00005D}, Example 16.2). We introduce the following
notation:
\[
W=\omega_{\mathcal{A}_{x}/k}\subset H=H_{dR}^{1}(\mathcal{A}_{x}),\,\,\,\,H_{0}=H_{dR}^{1}(\mathcal{A}_{x})[V]
\]
and

\[
P_{0}=\mathcal{P}_{0,y}\subset H_{0}\cap W.
\]
These are $k$-vector spaces with $\kappa$ action. The polarization
pairing $\left\{ ,\right\} _{\phi}:H\times H\to k$ induces a perfect
pairing 
\[
\left\{ ,\right\} _{\phi}:W\times H/W\to k,
\]
satisfying $\{\iota(a)u,v\}_{\phi}=\{u,\iota(\overline{a})v\}_{\phi}.$
We claim that $S^{\sharp}(k[\epsilon])_{y}$ is identified with
\[
\left\{ (\varphi,\psi)|\,\,\varphi\in\Hom_{\kappa}(W,H/W)^{sym},\,\,\psi\in\Hom_{\kappa}(P_{0},H_{0}/P_{0}),\,\,\varphi|_{P_{0}}=\psi\mod W\right\} .
\]
Indeed, by Grothendieck's crystalline deformation theory \cite{=00005BGro=00005D},
$S(k[\epsilon])_{x}$ is identified with $\Hom_{\kappa}(W,H/W)^{sym}.$
The superscript $sym$ refers to homomorphisms symmetric with respect
to $\{,\}_{\phi},$ i.e. satisfying $\{w,\varphi(w')\}_{\phi}=\{w',\varphi(w)\}_{\phi}$
for all $w,w'\in W$. 

The space $\Hom_{\kappa}(P_{0},H_{0}/P_{0})$ classifies infinitesimal
deformations of $P_{0}$ preserving the type $\Sigma$ and the property
of being killed by $V.$ This is because under the canonical identification
\[
H_{cris}^{1}(\mathcal{A}_{x})_{\Spec(k)\hookrightarrow\Spec(k[\epsilon])}=H_{dR}^{1}(\mathcal{A}_{x})\otimes_{k}k[\epsilon]
\]
the map induced on the left hand side by functoriality from $\Ver:\mathcal{A}_{x}^{(p)}\to\mathcal{A}_{x}$
is $V_{cris}=V_{dR}\otimes1.$

Finally the condition $\varphi|_{P_{0}}=\psi\mod W$ means that the
infinitesimal deformation of $P_{0}$ stays in the Hodge filtration.

\medskip{}

Our problem is now reduced to linear algebra. Note first that
\[
\varphi\in\Hom_{\kappa}(W,H/W)^{sym}=\Hom(P,H(\Sigma)/P),
\]
where we have written $P=W(\Sigma),$ the symmetry condition with
respect to the pairing $\{,\}_{\phi}$ then determining uniquely the
component in $\Hom(Q,H(\overline{\Sigma})/Q),$ where $Q=W(\overline{\Sigma})$.
Likewise,
\[
\psi\in\Hom_{\kappa}(P_{0},H_{0}/P_{0})=\Hom(P_{0},H_{0}(\Sigma)/P_{0}).
\]

The dimension of $H_{0}(\Sigma)$ is $n.$ Indeed, $H_{0}=H[V]$ is
the image of the map $F:H^{(p)}\to H,$ whose kernel is $W^{(p)}$.
As $H$ itself is balanced (of type $(m+n,m+n)),$ and $H^{(p)}[F]=W^{(p)}$
is of type $(m,n)$, 
\[
H_{0}=H[V]\simeq H^{(p)}/W^{(p)}
\]
is of type $(n,m).$ Thus, $\psi$ varies in a space of dimension
$(n-m)m.$

Given $\psi,$ $\varphi|_{P_{0}}$ is determined, and by this we take
care of the constraint $\varphi|_{P_{0}}=\psi\mod W.$ It remains
to extend $\varphi$ from $P_{0}$ to $P.$ As the codimension of
$P_{0}$ in $P$ is $m$ and the dimension of $H(\Sigma)/P$ is $(n+m)-n=m,$
this adds $m^{2}$ dimensions to the tangent space. Altogether
\[
\dim S^{\sharp}(k[\epsilon])_{y}=(n-m)m+m^{2}=nm
\]
as desired.
\end{proof}

\subsection{\label{subsec:The-maps-from S_0(p)_m and S_0(p)_et}The maps from
$S_{0}(p)_{m}$ and $S_{0}(p)_{et}$ to $S^{\sharp}$}

We denote by $S_{0}(p)_{m}$ and $S_{0}(p)_{et}$ the Zariski closures
in $S_{0}(p)$ of $S_{0}(p)_{m}^{\ord}$ and $S_{0}(p)_{et}^{\ord}$.
Our purpose is to define \emph{finite flat} morphisms
\[
\pi_{m}^{\sharp}:S_{0}(p)_{m}\to S^{\sharp},\,\,\,\pi_{et}^{\sharp}:S_{0}(p)_{et}\to S^{\sharp}
\]
which extend the restrictions $\pi_{m}:S_{0}(p)_{m}^{\ord}\to S^{\ord}$
and $\pi_{et}:S_{0}(p)_{et}^{\ord}\to S^{\ord}$ of $\pi$ to the
ordinary-multiplicative and ordinary-\'etale loci. In fact, $\pi_{m}^{\sharp}$
will be an isomorphism, and our main interest will be in $\pi_{et}^{\sharp}.$ 

We stress that although the compositions of these maps with the projection
from $S^{\sharp}$ to $S$ both agree with $\pi$, as maps to $S^{\sharp}$
they do not agree on the intersection of $S_{0}(p)_{m}$ and $S_{0}(p)_{et}$,
except for the part lying over $S_{\sharp}.$

\subsubsection{The multiplicative component}
\begin{lem}
\label{lem:P_0}Let $R$ be a $\kappa$-algebra and $(\underline{A},H)\in S_{0}(p)_{m}(R)$.
Then: (i) $\Fr(H)=0$, (ii) $\omega_{H/R}(\Sigma)$ is locally free
of rank $m$, and (iii) the subsheaf 
\begin{equation}
\mathcal{P}_{0}:=\ker(\omega_{A/R}(\Sigma)\to\omega_{H/R}(\Sigma))\label{eq:P_0 on S_0(p)_m}
\end{equation}
agrees with $\mathcal{P}_{0}=\mathcal{P}[V]$ over $S^{\ord}$, is
locally free of rank $n-m$ and killed by $V=\Ver{}_{A/R}^{*}.$
\end{lem}

\begin{proof}
(i) This is a closed condition and it holds on $S_{0}(p)_{m}^{\ord}$,
so it holds by continuity on its Zariski closure $S_{0}(p).$ 

(ii) By reduction to the universal case we may assume, since $S_{0}(p)_{m}$
is reduced by Proposition \ref{prop:=00005BG=0000F6=00005D}, that
$R$ is reduced. It is therefore enough to prove that all the geometric
fibers of $\omega_{H/R}(\Sigma)$ are of the same dimension $m.$
We may therefore assume that $R=k$ is an algebraically closed field. 

Let $M=M(H)$ be the covariant Dieudonn\'e module of $H/k.$ Recall
that $\Lie(H)=M[V]$, where $V$ is the map $M(H)\to M(H^{(p)})$
induced by $\Fr.$ By (i) $\Lie(H)=M.$ But the Dieudonn\'e module is
$2m$-dimensional and balanced, so $M(\Sigma)$ is $m$-dimensional.
Hence, $\Lie(H)(\Sigma)$ and its dual $\omega_{H/k}(\Sigma)$ are
$m$-dimensional. 

(iii) Since the map $\omega_{A/R}\to\omega_{H/R}$ is surjective,
the assertion on the rank follows from (ii). The condition that $V$
kills $\mathcal{P}_{0}$ holds over $S_{0}(p)_{m}^{\ord}$ (where
$\mathcal{P}_{0}=\mathcal{P}[V]$), so being a closed condition, continues
to hold over $S_{0}(p)_{m}.$
\end{proof}
Define the map $\pi_{m}^{\sharp}:S_{0}(p)_{m}\to S^{\sharp}$ by
\[
\pi_{m}^{\sharp}(\underline{A},H)=(\underline{A},\mathcal{P}_{0}),
\]
where $\mathcal{P}_{0}$ is given by (\ref{eq:P_0 on S_0(p)_m}).
By the lemma, it is well defined, and it clearly extends the isomorphism
$\pi_{m}:S_{0}(p)_{m}^{\ord}\simeq S^{\ord}.$
\begin{prop}
\label{prop:pi_m^sharp}
The map $\pi_{m}^{\sharp}$ is an isomorphism $S_{0}(p)_{m}\simeq S^{\sharp}.$
\end{prop}

\begin{proof}
We first check that the map is 1-1 on $k$-points where $k$ is an
algebraically closed field. Let $(\underline{A},H)\in S_{0}(p)_{m}(k)$
and $\pi_{m}^{\sharp}(\underline{A},H)=(\underline{A},\mathcal{P}_{0}).$
The proof of Lemma \ref{lem:P_0} shows that $\Lie(H)(\Sigma)$ is
uniquely determined by $\mathcal{P}_{0}$ as the annihilator of $\mathcal{P}_{0}$
in $\Lie(A)(\Sigma).$ On the other hand $\Lie(H)(\overline{\Sigma})=\Lie(A)(\overline{\Sigma})$
since both are $m$-dimensional. We conclude that $\Lie(H)$ is uniquely
determined as a subspace of $\Lie(A)=M(A[p])[V].$ But the proof of
Lemma \ref{lem:P_0} also shows that $\Lie(H)=M(H)$, hence $M(H)\subset M(A[p])$
is uniquely determined, so $H\subset A[p]$ is uniquely determined
as a subgroup scheme.

Since $\pi_{m}^{\sharp}$ is clearly proper and quasi-finite, it is
finite. It is also birational. But $S^{\sharp}$ is smooth, so by
Zariski's Main Theorem $\pi_{m}^{\sharp}$ is an isomorphism.
\end{proof}
We let $\sigma_{m}^{\sharp}:S^{\sharp}\to S_{0}(p)_m \subset S_{0}(p)$ be the section
inverse to $\pi_{m}^{\sharp}.$

\subsubsection{The \'etale component}

We are now ready to extend the diagram which was constructed in \S\ref{subsec:Morphisms-between-ordinary-=0000E9tale}
from the ordinary locus to its Zariski closure.

\[\xymatrix{& S_0(p)_{et}^{(p)}\ar[d]^{\rho^\prime} \ar[dr]^{Fr_p} & \\ S_0(p)_{et}\ar[ru]^{Fr_p}\ar[r]^(0.6){\pi_{et}^{\sharp}}\ar[dr]_{\theta^\prime} & S^{\sharp}\ar@<-0.5ex>[d]_{\sigma_m^{\sharp}}\ar[r]^(0.4)\rho & S_0(p)_{et} \\ &  S_0(p)_{m} \ar@<-0.5ex>[u]_{\pi_m^{\sharp}}\ar[ur]_\theta }\]

\medskip{}
\begin{thm}
\label{thm:Extended theta and rho}(i) The maps $\rho,\rho',\theta$
and $\theta'$ can be extended, and a map $\pi_{et}^{\sharp}$ can
be defined, so that the diagram above is commutative. (ii) All the
morphisms in the diagram are finite and surjective. The maps $\pi_{et}^{\sharp}$,
$\theta'$ and $\rho'$ are also flat. 
\end{thm}

This theorem can be strengthened, as we shall see in Theorem \ref{thm:smoothness_S_0(p)_et}
below: The maps $\rho$ and $\theta$ are also flat, and $S_{0}(p)_{et}$
is non-singular. However, this will result from considerations involving
the extension of the foliation $\mathcal{T}S^{+}$ to $S^{\sharp}$
and not merely from the constructions outlined here.

In the proof we shall use repeatedly the fact that if $H_{1}$ and
$H_{2}$ are two finite flat subgroup schemes of a finite flat group
scheme $G$ over a separated base, then the locus in the base where
$H_{1}\subset H_{2}$ is closed.
\begin{proof}
We first extend our map $\theta$ (\emph{cf.} \S\ref{subsec:Morphisms-between-ordinary-=0000E9tale})
from the ordinary locus to a morphism
\[
\theta:S_{0}(p)_{m}\to S_{0}(p)_{et},\,\,\,\,\theta(\underline{A},H)=(\underline{A}^{(p^{2})},\Fr(\Ver^{-1}(H))).
\]
We must show that this is well-defined. We have proved above that
for $(\underline{A},H)\in S_{0}(p)_{m}(R),$ $\Fr(H)=0,$ which means
$p\Ver^{-1}(H)=0,$ or $\Ver^{-1}(H)\subset A^{(p)}[p].$ Thus $\Ver^{-1}(H)$
is a finite flat subgroup scheme of rank $p^{n+3m}$ of $A^{(p)}[p].$

We claim that $A^{(p)}[\Fr]\subset\Ver^{-1}(H).$ This holds, as we
have seen before, over $S_{0}(p)_{m}^{\ord},$ so by the remark preceding
the proof, it persists over the Zariski closure $S_{0}(p)_{m}$.

We may now conclude that $J=\Fr(\Ver^{-1}(H))$ is a finite flat subgroup
scheme of rank $p^{2m}$ of $A^{(p^{2})}[p].$ That it is isotropic
follows from the fact that $\Ver^{-1}(H)\subset A^{(p)}[p],$ because
for $u,v\in\Ver^{-1}(H)$
\[
\left\langle \Fr(u),\Fr(v)\right\rangle _{A^{(p^{2})}}=\left\langle u,\Ver\circ\Fr(v)\right\rangle _{A^{(p)}}=\left\langle u,pv\right\rangle _{A^{(p)}}=0.
\]
Clearly $J$ is $\mathcal{O}_{E}$-stable. To check that it is Raynaud
we may assume, as usual, that $R=k$ is an algebraically closed field.
The exact sequences of covariant Dieudonn\'e modules
\[
0\to M(A^{(p)}[\Ver])\to M(\Ver^{-1}(H))\to M(H)\to0
\]
and
\[
0\to M(A^{(p)}[\Fr])\to M(\Ver^{-1}(H))\to M(J)\to0
\]
show, since $M(H)$ is balanced and $M(A^{(p)}[\Ver])$ and $M(A^{(p)}[\Fr])$
have the same signature, that $J$ is Raynaud. To see this last point,
from the exact sequence
\[
0\to M(A[\Fr])\to M(A[p])\to M(A^{(p)}[\Ver])\to0
\]
and the fact that $M(A[p])$ is balanced, we get that the types of
$M(A[\Fr])$ and $M(A^{(p)}[\Ver])$ are opposite, hence the types
of $M(A^{(p)}[\Fr])$ and $M(A^{(p)}[\Ver])$ are the same.

We conclude that $\theta$ is well-defined and maps $S_{0}(p)_{m}$
into $S_{0}(p).$ Since it maps $S_{0}(p)_{m}^{\ord}$ into $S_{0}(p)_{et}^{\ord},$
it actually maps $S_{0}(p)_{m}$ into $S_{0}(p)_{et}.$

As before, we define
\[
\rho:S^{\sharp}\to S_{0}(p)_{et},\,\,\,\,\rho=\theta\circ\sigma_{m}^{\sharp}.
\]

We shall next define a similar extension of $\theta'$ (\emph{cf.}
\S\ref{subsec:Morphisms-between-ordinary-=0000E9tale}) to a map 
\[
\theta':S_{0}(p)_{et}\to S_{0}(p)_{m},\,\,\,\theta'(\underline{A},H)=(\underline{A},\Ver(\Fr^{-1}(H^{(p^{2})}))),
\]
and let
\[
\pi_{et}^{\sharp}=\pi_{m}^{\sharp}\circ\theta':S_{0}(p)_{et}\to S^{\sharp}.
\]
Let $(\underline{A},H)\in S_{0}(p)_{et}(R).$ Consider $H'=\Ver(\Fr^{-1}(H^{(p^{2})}))\subset A.$
We claim that $H'$ is a finite flat subgroup scheme of $A[p]$ of
rank $p^{2m}.$ To see it, note first that $\Fr^{-1}(H^{(p^{2})})$
is finite flat of rank $p^{n+3m},$ being the kernel of the isogeny
\[
\psi:A^{(p)}\overset{\Fr}{\to}A^{(p^{2})}\to A^{(p^{2})}/H^{(p^{2})}.
\]
Second, note that $\Fr^{-1}(H^{(p^{2})})$ is contained in $A^{(p)}[p].$
Indeed, this holds over $S_{0}(p)_{et}^{\ord},$ so it holds by continuity
over the whole of $S_{0}(p)_{et}.$ Third, we claim that
\[
A^{(p)}[\Ver]\subset\Fr^{-1}(H^{(p^{2})}).
\]
This too follows by continuity, since it clearly holds over $S_{0}(p)_{et}^{\ord}.$
We conclude that $H'$ is finite flat of rank $p^{2m}.$ Moreover
$\Fr(H')=0$, since $\Fr\circ\Ver=p\cdot id_{A^{(p)}}$ and $\Fr^{-1}(H^{(p^{2})})\subset A^{(p)}[p].$
One checks now, as before, that $H'$ is isotropic, $\mathcal{O}_{E}$-stable
and Raynaud. Setting $\theta'(\underline{A},H)=(\underline{A},H')$
defines a map from $S_{0}(p)_{et}$ to $S_{0}(p)$. As it maps $S_{0}(p)_{et}^{\ord}$
to $S_{0}(p)_{m}^{\ord}$ its image is in $S_{0}(p)_{m}$ and $\theta'$
extends the morphism between the ordinary parts constructed in \S\ref{subsec:Morphisms-between-ordinary-=0000E9tale}.

This concludes the definition of the maps in the lower triangles.
It is easily checked that
\[
\theta\circ\theta'=\rho\circ\pi_{et}^{\sharp}=Fr_{p}^{2}.
\]
As $S^{\sharp}$ is non-singular, $S_{0}(p)_{et}$ is Cohen-Macaulay
and $\pi_{et}^{\sharp}$ is finite and onto, we deduce from \cite{=00005BEis=00005D}
18.17 that $\pi_{et}^{\sharp}$ is flat. Hence $\theta'=\sigma_{m}^{\sharp}\circ\pi_{et}^{\sharp}$
is also flat.

It remains to define $\rho'$. This has been done over the ordinary
locus, via the modular interpretation, in \S\ref{subsec:Morphisms-between-ordinary-=0000E9tale}.
Extending the definition of $\rho'$ via the modular interpretation
to $S_{0}(p)_{et}^{(p)}$ is possible, but painful. Instead, we conclude
the proof of the theorem with the help of the following general lemma.
It follows from it that $\rho'$ extends to the whole of $S_{0}(p)_{et}^{(p)}$.
The commutativity of the diagram follows by continuity from the fact
that it is commutative over the ordinary locus. The fact that $\rho'$
is flat follows again from \cite{=00005BEis=00005D} 18.17 since $S_{0}(p)_{et}^{(p)}$
is Cohen-Macaulay, $S^{\sharp}$ is non-singular and $\rho'$ is finite
and onto. Surjectivity of the maps follows from the fact that they
are finite and dominant.
\end{proof}
\begin{lem}
Let $X$ and $Y$ be irreducible varieties over a perfect field $\kappa$
of characteristic $p,$ with $Y$ normal. Suppose that we are given
a finite morphism $\rho$ and a rational map $\rho'$
\[
Y\overset{\rho'}{\dashrightarrow}X\overset{\rho}{\to}Y^{(p)}
\]
such that $\rho\circ\rho'=Fr_{p}.$ Then $\rho'$ extends to a morphism
on the whole of Y.
\end{lem}

\begin{proof}
We may assume that $X$ and $Y$ are affine. Let $X=\Spec(A)$ and
$Y=\Spec(B)$. Let $K$ and $L$ be the fields of fractions of $A$
and $B$, respectively. Let $B^{(p)}\subset L^{(p)}$ be the image
of $B$ and $L$ under Frobenius. Then $B^{(p)}\subset A$ and $L^{(p)}\subset K\subset L.$
Since $A$ is integral over $B^{(p)}$ and $B$ is integrally closed
in $L,$ we get $A\subset B,$ which proves the lemma, and concludes
the proof of Theorem \ref{thm:Extended theta and rho}. 
\end{proof}

\subsection{The extended foliation}

Via the map $f:S^{\sharp}\to S$ we can pull back the universal abelian
scheme $\mathcal{A}$ and its de Rham cohomology, and get a locally
free sheaf $f^{*}H_{dR}^{1}(\mathcal{A}/S)$ endowed with its own
Gauss-Manin connection
\[
\nabla:f^{*}H_{dR}^{1}(\mathcal{A}/S)\to f^{*}H_{dR}^{1}(\mathcal{A}/S)\otimes\Omega_{S^{\sharp}}^{1}.
\]
Over $S^{\sharp}$ we find the tautological vector bundle
\[
\mathcal{P}_{0}\subset f^{*}\mathcal{P}\subset f^{*}\omega_{\mathcal{A}/S}\subset f^{*}H_{dR}^{1}(\mathcal{A}/S).
\]
Let $y\in S^{\sharp}$. In the proof of Theorem \ref{thm:smoothness of S=000023}
we identified each tangent vector $\eta\in\mathcal{T}S_{y}^{\sharp}$
with a pair $\eta=(\varphi,\psi)$ such that, in the notation introduced
there,
\[
\varphi\in\Hom_{\kappa}(W,H/W)^{sym},\,\,\psi\in\Hom_{\kappa}(P_{0},H_{0}/P_{0}),\,\,\varphi|_{P_{0}}=\psi\mod W.
\]
We define the subsheaf $\mathcal{T}S^{\sharp+}\subset\mathcal{T}S^{\sharp}$
by the condition
\[
\eta\in\mathcal{T}S_{y}^{\sharp+}\Leftrightarrow\psi=0.
\]
\begin{prop}
\label{prop:extension_of_foliation}
The subsheaf $\mathcal{T}S^{\sharp+}$ is a height 1 foliation of
rank $m^{2}$ which extends $\mathcal{T}S^{+}$. It is transversal
to the fibers of $f:S^{\sharp}\to S.$
\end{prop}

\begin{proof}
In the proof of Theorem \ref{thm:smoothness of S=000023} we found
that $\dim\mathcal{T}S_{y}^{\sharp+}=m^{2}.$ As the base $S^{\sharp}$
is reduced and the dimensions of its fibers are constant, $\mathcal{T}S^{\sharp+}$
is a sub-bundle of rank $m^{2}$. The tangent space to the fiber of
$f$ through $y$ is the set of pairs $(\varphi,\psi)$ with $\varphi=0.$
Thus $\mathcal{T}S^{\sharp+}$ is transversal to it. 

If $y\in S^{\ord}$, Corollary \ref{cor: characterization of TS+}
shows that $\psi=0$ is equivalent to $\eta\in\mathcal{T}S^{+}.$
Finally, the fact that $\mathcal{T}S^{\sharp+}$ is a $p$-Lie subalgebra
follows by continuity from the fact that $\mathcal{T}S^{+}$ is closed
under Lie bracket and raising to power $p,$ since $S^{\ord}$ is
dense in $S^{\sharp}.$
\end{proof}
We can now state the main theorem of this section.
\begin{thm}
\label{thm:smoothness_S_0(p)_et}The variety $S_{0}(p)_{et}$ is non-singular,
the morphism $\rho$ is finite and flat, and identifies $S_{0}(p)_{et}$
with the quotient of $S^{\sharp}$ by the foliation $\mathcal{T}S^{\sharp+}.$
\end{thm}

\begin{proof}
By Proposition \ref{prop:=00005BEk=00005D,-Proposition-2.4.} we know
that $\mathcal{T}S^{\sharp+}$ corresponds to a finite flat quotient
map
\[
S^{\sharp}\overset{\widetilde{\rho}}{\to}S_{0}(p)_{et}^{\sim}
\]
onto a non-singular variety $S_{0}(p)_{et}^{\sim},$ which, thanks
to Theorem \ref{thm:fol-mor corresp}, coincides with $S_{0}(p)_{et}$
over $S^{\ord}.$ Denoting, for simplicity, $X=S^{\sharp}$, $Y=S_{0}(p)_{et}$
and $Y^{\sim}=S_{0}(p)_{et}^{\sim}$ we get (by definition) that $\mathcal{O}_{Y^{\sim}}$
is the subsheaf of $\mathcal{O}_{X}$ killed by $\mathcal{T}X^{+}.$
If $s$ is a section of $\mathcal{O}_{Y}$ over a Zariski open $U$
and $\xi\in\mathcal{T}X^{+}(U)$ then over $U\cap Y^{\ord}$ $s$
is killed by $\xi,$ hence by continuity $\xi s=0$ on all of $U$.
This shows $\mathcal{O}_{Y}\subset\mathcal{O}_{Y^{\sim}}\subset\mathcal{O}_{X}$
so the morphism $\rho$ factors as $\sigma\circ\widetilde{\rho}$
for a unique finite birational morphism $\sigma:S_{0}(p)_{et}^{\sim}\to S_{0}(p)_{et}.$
However, according to Proposition \ref{prop:=00005BG=0000F6=00005D},
$S_{0}(p)_{et}$ is normal. Zariski's Main Theorem implies now that
$\sigma$ is an isomorphism, completing the proof.
\end{proof}

\section{Integral subvarieties}

Recall that an integral subvariety of the foliation $\mathcal{T}S^{\sharp+}$
is a non-singular subvariety $Y\subset S^{\sharp}$ for which $\mathcal{T}S^{\sharp+}|_{Y}=\mathcal{T}Y$.
In this section we find two types of integral subvarieties: Shimura
varieties of signature $(m,m)$ embedded in $S$, and the EO stratum
$S_{\fol}.$ We end the paper with the natural question whether these
are the only global integral subvarieties.

\subsection{Shimura subvarieties of signature $(m,m)$}

There are many ways to embed Shimura varieties associated with unitary
groups of signature $(m,m)$ in our unitary Shimura variety $S_{K}$.
These smaller Shimura varieties can be associated with a quasi-split
unitary group, or with an inner form of it. The embeddings extend
to the integral models, hence to their special fibers, and can be
described in terms of the respective moduli problems. For $(n,m)=(2,1)$
and the resulting embeddings of modular curves or Shimura curves in
Picard modular surfaces, see \cite{=00005BdS-G1=00005D} \S4.2.2 or
\cite{=00005BdS-G2=00005D} \S1.4.
\begin{thm}
Let $S'$ be the special fiber of a unitary Shimura variety of signature
$(m,m)$ embedded in $S.$ Then $S'\cap S^{\ord}$ is an integral
subvariety of $\mathcal{T}S^{+}.$
\end{thm}

\begin{proof}
The proof of Theorem 2.3(ii) in \cite{=00005BdS-G2=00005D} can be
easily generalized, once the embedding of the appropriate moduli problems
is written down explicitly. A different approach is to use Theorem
\ref{thm:compatibility with Moonen}. The set $S'\cap S^{\ord}$ is
open and dense in $S'$ and the $2m$-dimensional abelian varieties
which it parametrizes are ordinary (in the usual sense). The classical
Serre-Tate theorem attaches a structure of a formal torus to the formal
neighborhood $\widehat{S'_{x}}=Spf(\mathcal{\widehat{O}}_{S',x})$
in $S'$ of a point $x\in S'\cap S^{\ord}$. The compatibility of
Moonen's generalized Serre-Tate coordinates under embeddings of Shimura
varieties shows that under the embedding $\iota$ of $S'$ in $S$
the formal neighborhood $\widehat{S'_{x}}$ gets mapped to $\widehat{T}_{\iota(x)}\subset Spf(\widehat{\mathcal{O}}_{S,\iota(x)}).$
The theorem follows now from Theorem \ref{thm:compatibility with Moonen}.
\end{proof}

\subsection{EO strata}

The proof that the EO stratum $S_{\fol}$ is an integral subvariety
of the foliation $\mathcal{T}S^{+}$ is more difficult. We follow
the strategy outlined in \cite{=00005BdS-G1=00005D}, \S3.4, in particular
Lemma 3.10 there, but the generalization from signature $(2,1)$ to
the general case requires some work.

Recall that we denoted by $S_{\sharp}$ the open subset of $S$ where
$f:S^{\sharp}\to S$ is an isomorphism, and that $S_{\fol}$ is the
unique minimal EO stratum in $S_{\sharp},$ so we are justified in
writing $\mathcal{T}S^{+}$ instead of $\mathcal{T}S^{\sharp+}$ when
we refer to $\mathcal{T}S^{\sharp+}|_{S_{\fol}}.$ Recall also that
$\dim(S_{\fol})=m^{2}=\rk(\mathcal{T}S^{+}),$ a hint that we are
on the right track.

\subsubsection{The Dieudonn\'e module at a point of $S_{\protect\fol}$}

The following Proposition describes the structure of the \emph{contravariant
}Dieudonn\'e module $D_{0}=D(\mathcal{A}_{x}[p])$ at a point $x\in S_{\fol}(k)$
($k,$ as usual, algebraically closed and containing $\kappa$). It
can be deduced from \cite{=00005BMo2=00005D}, \S4.9, see also \cite{=00005BWoo=00005D},
\S3.5. Recall that there exists a canonical identification
\[
D_{0}=H_{dR}^{1}(\mathcal{A}_{x}/k),
\]
and that the skew-symmetric pairing $\{,\}_{\phi}$ on $D_{0}$, induced
by the polarization, becomes under this identification the pairing
$\{x,y\}_{\phi}=\{x,(\phi^{-1})^{*}y\}$ where $\{,\}$ is the canonical
pairing on $H_{dR}^{1}(\mathcal{A}_{x}/k)\times H_{dR}^{1}(\mathcal{A}_{x}^{t}/k)$.
\begin{prop}
\label{prop:DMfol}Let $x\in S_{\fol}(k)$ and $D_{0}=D(\mathcal{A}_{x}[p]).$
There exists a basis 
\[
\{e_{1},\dots,e_{n+m},f_{1},\dots,f_{n+m}\}
\]
 of $D_{0}$ with the following properties.

(i) $\kappa$ acts on the $e_{i}$ via $\Sigma$ and on the $f_{j}$
via $\overline{\Sigma}$.

(ii) $\{e_{i},f_{n+m+1-i}\}_{\phi}=1,$ and the other $\{e_{i},f_{j}\}_{\phi}$,
as well as $\{e_{i},e_{j}\}_{\phi}$ and $\{f_{i},f_{j}\}_{\phi}$,
are all 0.

(iii) The maps $F:D_{0}^{(p)}\to D_{0}$ and $V:D_{0}\to D_{0}^{(p)}$
induced by $\Fr$ and $\Ver$ are given by the following tables. We
abbreviate the list $e_{a},\dots,e_{b}$ as $e_{[a,b]}$ and similarly
with $f_{[a,b]}$. 

\bigskip{}

\renewcommand{\arraystretch}{1.5} \begin{tabular}{|p{1.4cm}||p{2.4cm}|p{2.4cm}|p{2.4cm}|} \hline $i\in $ & $[1, n-m]$ & $[n-m+1, n]$& $[n+1, n+m]$ \\ \hline \hline $F(e_i^{(p)})$ & $0$ & $-f_{i-n+m}$ & $0$ \\ \hline $V(e_i)$ & $0$ & $0$ & $f_{i-n+m}^{(p)}$  \\ \hline \end{tabular}

\bigskip{}

\begin{tabular}{|p{1.4cm}||p{2.4cm}|p{2.4cm}|p{2.4cm}| } \hline $j\in $ & $[1, m]$ & $[m+1, 2m]$& $[2m+1, n+m]$ \\ \hline \hline $F(f_j^{(p)})$ & $-e_j$ & $0 $ & $-e_{j-m}$ \\ \hline \end{tabular}

\bigskip{}

\begin{tabular}{|p{1.4cm}||p{2.4cm}|p{2.4cm}|p{2.4cm}| } \hline $j\in $ & $[1, m]$ & $[m+1, n]$& $[n+1, n+m]$ \\ \hline \hline $V(f_j)$ & $0$ & $e_{j-m}^{(p)}$ & $e_j^{(p)}$ \\ \hline \end{tabular} 

\bigskip{}
In particular, $\omega_{\mathcal{A}_{x}/k}=(D_{0}^{(p)}[F])^{(p^{-1})}=Span_{k}\{e_{[1,n-m]},e_{[n+1,n+m]},f_{[m+1,2m]}\}$
and $\ker(V_{\mathcal{P}})_{x}=\mathcal{P}_{0,x}=Span_{k}\{e_{[1,n-m]}\}.$
\end{prop}

\begin{cor}
\label{cor:V_O(Q)_x at S_fol}Notation as above, if $x\in S_{\fol}$,
then $V(\mathcal{Q})_{x}=Span_{k}\{e_{1}^{(p)},\dots,e_{m}^{(p)}\}$
if $2m\le n$, and $V(\mathcal{Q})_{x}=Span_{k}\{e_{1}^{(p)},\dots,e_{n-m}^{(p)},e_{n+1}^{(p)},\dots,e_{2m}^{(p)}\}$
if $n<2m$.
\end{cor}

In passing, we note that the Hasse matrix $H_{\mathcal{A}/S}=V_{\mathcal{P}}^{(p)}\circ V_{\mathcal{Q}}=0$
over $S_{\fol}$ if $2m\le n$ but not if $n<2m$. The Hasse invariant
$h_{\mathcal{A}/S}=\det(H{}_{\mathcal{A}/S})$ always vanishes, of
course.

\subsubsection{Proof of the main theorem}

In the proof of the following theorem we shall separate the three
cases (i) $n=2m$ (ii) $2m<n$ (iii) $n<2m$. Although the idea of
the proof is the same, the three cases become progressively more complicated.
Thus, for the sake of exposition, we felt it was better to treat them
separately, at the price of some repetition.

As a matter of notation, if $f:T\to S$ is a morphism of schemes,
and $\mathcal{F}$ is a coherent sheaf on $S$, we denote by $\mathcal{F}(T)$
the global sections of $f^{*}\mathcal{F}$ on $T$. We shall employ
this notation in particular when $T$ is an infinitesimal neighborhood
of a closed point of $S$, or a closed subscheme of such an infinitesimal
neighborhood.
\begin{thm}
\label{thm:Sfol_integral_submanifold}The EO stratum $S_{\fol}$ is
an integral subvariety of the foliation $\mathcal{T}S^{+},$ i.e.
$\mathcal{T}S^{+}|_{S_{\fol}}=\mathcal{T}S_{\fol}.$ 
\end{thm}

\begin{proof}
\textbf{1}. Let $x\in S_{\fol}$ and $R=\mathcal{O}_{S,x}/\mathfrak{m}_{S,x}^{2}$.
Let $D=H_{dR}^{1}(\mathcal{A}/R)$ be the infinitesimal deformation
of $D_{0}$. Although $\Spec(R)$ is not smooth over $k$, the $R$-module
$D$ inherits, by base-change, the Gauss-Manin connection 
\[
\nabla:D\to D\otimes_{R}(R\otimes_{\mathcal{O}_{S}}\Omega_{S}^{1}).
\]
As it admits a basis of horizontal sections over $\Spec(R),$ we may
write $D=D_{0}\otimes_{k}R$, the horizontal sections being $D_{0}\otimes_{k}k=D_{0}.$
Since the Gauss-Manin connection is compatible with isogenies, $F$
and $V$ take horizontal sections to horizontal sections. Thus Proposition
\ref{prop:DMfol}(iii) holds also for $F:D^{(p)}\to D$ and $V:D\to D^{(p)}.$
The pairing $\{,\}_{\phi}$ is horizontal for the Gauss-Manin connection,
i.e.
\[
d\{x,y\}_{\phi}=\{\nabla x,y\}_{\phi}+\{x,\nabla y\}_{\phi},
\]
so the formulae from part (ii) of the Proposition also persist in
$D$. What \emph{does change}, and, according to Grothendieck, completely
determines the infinitesimal deformation, is the Hodge filtration
$\omega_{\mathcal{A}/R}.$ One sees that the most general deformation
of $\omega_{\mathcal{A}_{x}/k}$ is given by

\begin{multline} \omega_{\mathcal A/R}(\Sigma) = \text{Span}_R \{ e_i + \sum_{j = n-m+1}^n u_{ij}e_j, \;\; e_\ell + \sum_{j = n-m+1}^n v_{\ell j} e_j  \;:  \\ 1 \leq i \leq n-m, \; n+1 \leq \ell \leq n+m \}, \end{multline}where
the $nm=(n-m)m+m^{2}$ variables $u_{ij}$ and $v_{\ell j}$ are local
parameters at $x,$ and their residues modulo $\mathfrak{m}_{S,x}^{2}$
form a basis for $\mathcal{T}S_{x}^{\vee}.$ The deformation $\omega_{\mathcal{A}/R}(\overline{\Sigma})$
is then completely determined by $\omega_{\mathcal{A}/R}(\Sigma)$
and the condition that $\omega_{\mathcal{A}/R}$ is isotropic for
$\{,\}_{\phi}.$ A small computation reveals that it is given by

\begin{multline} \omega_{\mathcal A/R}(\overline \Sigma) = \text{Span}_R \{ f_{n+m+1 - j} - \sum_{i =1}^{n-m} u_{ij}f_{n+m+1 - i} - \sum_{\ell = n+1}^{n+m} v_{\ell j}f_{n+m + 1 - \ell} \;:  \\  n-m+1 \leq j \leq n \}. \end{multline} Compare
the proof of Theorem \ref{thm:smoothness of S=000023}. The data encoded
in the matrices $u$ and $v$ is just the data denoted there by
\[
\varphi\in\Hom_{\kappa}(W,H/W)^{sym}=\Hom(P,H(\Sigma)/P).
\]

\medskip{}

\textbf{2. }Consider the abelian scheme $\mathcal{A}^{(p)}$ over
$\Spec(R),$ and note that it is constant:
\[
\mathcal{A}^{(p)}=\Spec(R)\times_{\Phi_{R},\Spec(R)}\mathcal{A}=\Spec(R)\times_{\Spec(k)}\mathcal{A}_{x}^{(p)},
\]
since the absolute Frobenius $\phi_{R}$ of the ring $R$ factors
as
\begin{equation}
R\twoheadrightarrow k\overset{\phi_{k}}{\to}k\hookrightarrow R.\label{eq:FrobFactorization}
\end{equation}
Inside $D^{(p)}=H_{dR}^{1}(\mathcal{A}^{(p)}/R)=R\otimes_{\phi,R}D$
we therefore get
\[
\omega_{\mathcal{A}^{(p)}/R}(\overline{\Sigma})=\omega_{\mathcal{A}/R}(\Sigma)^{(p)}=Span_{R}\{e_{i}^{(p)},e_{\ell}^{(p)};\,1\le i\le n-m,\,n+1\le\ell\le n+m\}.
\]
In particular, 
\[
\mathcal{P}_{0}^{(p)}(\Spec(R))=R\otimes_{k}\mathcal{P}_{0,x}^{(p)}=Span_{R}\{e_{i}^{(p)};\,1\le i\le n-m\}.
\]

\medskip{}

\textbf{3. }Let us compute the image of a typical generator of $\mathcal{Q}(\Spec(R))=\omega_{\mathcal{A}/R}(\overline{\Sigma})$
under $V.$ 

\begin{equation}
V(f_{n+m+1-j}-\sum_{i=1}^{n-m}u_{ij}f_{n+m+1-i}-\sum_{\ell=n+1}^{n+m}v_{\ell j}f_{n+m+1-\ell})\equiv\label{eq:V-image}
\end{equation}
\[
-\sum_{i=1}^{\min(m,n-m)}u_{ij}e_{n+m+1-i}^{(p)}-\sum_{i=\min(m,n-m)+1}^{n-m}u_{ij}e_{n+1-i}^{(p)}\mod R\otimes V(\mathcal{Q})_{x}.
\]
By Corollary \ref{cor:V_O(Q)_x at S_fol} this image is contained
in $R\otimes V(\mathcal{Q})_{x}$ if and only if all $u_{ij}=0.$

\medskip{}

\textbf{4. }If $n=2m$ we can finish the proof as follows. Corollary
\ref{cor:V_O(Q)_x at S_fol} implies that over $S_{\fol}$ we then
have $V(\mathcal{Q})=\mathcal{P}_{0}^{(p)},$ because the same holds
at every $k$-valued point of $S_{\fol}$ and the base is reduced.
Let $R_{\fol}$ be the quotient of $R$ defined by
\[
\Spec(R_{\fol})=\Spec(R)\cap S_{\fol}.
\]
We get
\[
V(\mathcal{Q})(\Spec(R_{\fol}))=\mathcal{P}_{0}^{(p)}(\Spec(R_{\fol}))=R_{\fol}\otimes_{k}\mathcal{P}_{0,x}^{(p)}=R_{\fol}\otimes_{k}V(\mathcal{Q})_{x}.
\]
By Point 3 this means that over $\Spec(R_{\fol})$ we must have all
$u_{ij}=0.$ A dimension count shows that $\{u_{ij}=0\}$ is actually
the set of equations defining $S_{\fol}$ infinitesimally, i.e. $R_{\fol}=R/(u_{ij}).$
Thus $\mathcal{T}S_{\fol,x}$ is spanned by $\{\partial/\partial v_{\ell j};\,n+1\le\ell\le n+m,\,\,n-m+1\le j\le n\}$. 

On the other hand, from the explicit description of $\mathcal{P}(\Spec(R))=\omega_{\mathcal{A}/R}(\Sigma)$,
and from the characterization of $\mathcal{T}S^{+}$ given in Corollary
\ref{cor: characterization of TS+} (which holds, with the same proof, over all of $S_\sharp$), we find that $\mathcal{T}S_{x}^{+}$
is also spanned by $\partial/\partial v_{\ell j}$ ($n+1\le\ell\le n+m,\,n-m+1\le j\le n)$.
Indeed, $\mathcal{P}_{0}(\Spec(R))$ is spanned over $R$ by the sections
\[
e_{i}+\sum_{j=n-m+1}^{n}u_{ij}e_{j},
\]
which are killed by $\nabla_{\partial/\partial v_{\ell j}}$, but
are sent to sections which are outside $\mathcal{P}_{0}(\Spec(R))$
by $\nabla_{\partial/\partial u_{ij}}$. Thus, given $\xi\in\mathcal{T}S_{x}$,
$\nabla_{\xi}$ preserves $\mathcal{P}_{0}(\Spec(R))$ if and only
if $\xi$ is a linear combination of the $\partial/\partial v_{\ell j}$.
This concludes the proof of the theorem when $n=2m.$

\medskip{}

\textbf{5. }To finish the proof under the more general assumption
$2m\le n$ we must show that $\{u_{ij}=0\}$ is \emph{always} the
system of infinitesimal equations for $S_{\fol}.$ For that it is
enough to prove the following claim, that generalizes what we have
found for $n=2m.$ Let 
\[
\mathcal{D}=H_{dR}^{1}(\mathcal{A}/S),
\]
endowed with endomorphisms by $\mathcal{O}_{E}$ and the bilinear
form $\{,\}_{\phi}.$ This $\mathcal{D}$ is the unitary Dieudonn\'e
space of $\mathcal{A}[p]$ over $S,$ in the sense of \cite{=00005BWe=00005D}
(5.5). It is a locally free $\mathcal{O}_{S}$-module of rank $2(n+m)$,
$\omega_{\mathcal{A}/S}$ is a maximal isotropic sub-bundle, and $D=\mathcal{D}(\Spec(R))$
is the base-change of $\mathcal{D}$ under $\mathcal{O}_{S}\to\mathcal{O}_{S,x}/\mathfrak{m}_{S,x}^{2}=R.$

\medskip{}
\textbf{Claim. }\emph{Let $2m\le n.$ Over} $S_{\fol}$ \emph{there
is a sub-bundle} $\mathcal{M}\subset\mathcal{D}$ \emph{such that
at each geometric point} $x\in S_{\fol}(k)$, $V(\mathcal{Q})_{x}=\mathcal{M}_{x}^{(p)}\subset\mathcal{D}_{x}^{(p)}$.
(In fact, $\mathcal{M}$ will be a sub-bundle of $\omega_{\mathcal{A}/S_{\fol}}$.)

\medskip{}

Assuming the claim has been proved, we proceed as in the case $n=2m,$
when we identified $\mathcal{M}$ with $\mathcal{P}_{0}$. As the
base is reduced, $V(\mathcal{Q})=\mathcal{M}^{(p)}$. Since the absolute
Frobenius $\phi_{R}$ of the ring $R$ factors as in (\ref{eq:FrobFactorization})
and similarly for its quotient ring $R_{\fol}$, we get 
\[
\mathcal{M}^{(p)}(\Spec(R_{\fol}))=R_{\fol}\otimes_{k}\mathcal{M}_{x}^{(p)},
\]
as sub-modules of $\mathcal{D}^{(p)}(\Spec(R_{\fol}))=R_{\fol}\otimes_{k}D_{0}^{(p)}$. 

This means that $V(\mathcal{Q})(\Spec(R_{\fol}))$ lies in the subspace
$R_{\fol}\otimes V(\mathcal{Q})_{x}$. But we have seen that a typical
generator of $\mathcal{Q}(\Spec(R_{\fol}))$  maps to a vector outside
$R_{\fol}\otimes_{k}V(\mathcal{Q})_{x}$, unless all $u_{ij}=0.$
We conclude that $R_{\fol}=R/(u_{ij})$ as before.

\medskip{}

\textbf{6.}\textbf{\emph{ }}\emph{Proof of Claim. }The key for proving
the claim is the observation that if $2m\le n$
\[
V(\mathcal{Q})_{x}=\{F(D_{0}[V](\overline{\Sigma})^{(p)})\}^{(p)}.
\]
Indeed,
\[
D_{0}[V](\overline{\Sigma})=Span_{k}\{f_{1},\dots,f_{m}\},
\]
\[
F(Span_{k}\{f_{1}^{(p)},\dots,f_{m}^{(p)}\})=Span_{k}\{e_{1},\dots,e_{m}\},
\]
so we may use Corollary \ref{cor:V_O(Q)_x at S_fol}. Now $F(D_{0}[V]^{(p)})$
is part of the \emph{canonical filtration} of $D_{0}$ in the sense
of \cite{=00005BMo2=00005D} 2.5 (the part commonly denoted ``$FV^{-1}(0)$'').
It is therefore the (contravarient) Dieudonn\'e module of $\mathcal{A}_{x}[p]/\mathcal{N}_{x}$
for a certain subgroup scheme $\mathcal{N}_{x}$ of $\mathcal{A}_{x}[p]$
which belongs to the \emph{canonical filtration }of the latter, in
the sense of \cite{=00005BOo=00005D} (2.2). 

The point is that over any EO stratum, in particular over $S_{\fol}$,
the canonical filtration of $\mathcal{A}[p]$ \emph{exists} as a filtration
by finite flat subgroup schemes, and yields the canonical filtration
at each geometric point by specialization. See Proposition (3.2) in
\cite{=00005BOo=00005D}. Thus the $\mathcal{N}_{x}$ are the specializations
of a finite flat group scheme $\mathcal{N}$ over $S_{\fol}$. Letting
$\mathcal{M}$ be the $\Sigma$-part of the Dieudonn\'e module of $\mathcal{A}[p]/\mathcal{N}$
proves the claim. Alternatively, we can define $\mathcal{M}$ directly
as
\[
\mathcal{M}=F(\mathcal{D}[V](\overline{\Sigma})^{(p)})
\]
and use the constancy of fiber ranks over the reduced base $S_{\fol}$
to show that this is a sub-bundle of $\mathcal{D}.$

\medskip{}

\textbf{7.} \emph{The case $n<2m$}. The key idea when $2m\le n$ was the observation that
over $S_{\fol}$ the sub-bundle $V(\mathcal{Q})\subset\mathcal{D}^{(p)}$
was of the form $\mathcal{M}^{(p)}$ for a sub-bundle $\mathcal{M}\subset\mathcal{D}$.
This $\mathcal{M}$ was obtained as the $\Sigma$-part of a certain
piece in the canonical filtration of $\mathcal{D},$ namely $\mathcal{M}=FV^{-1}(0)(\Sigma).$
No such piece of the canonical filtration works if $n<2m.$ We are
able however to replace the equality $V(\mathcal{Q})=\mathcal{M}^{(p)}$
by an \emph{inclusion} $V(\mathcal{Q})\subset\mathcal{M}^{(p)}$ for
a carefully chosen $\mathcal{M}$, and modify the arguments accordingly. Assume therefore that $n<2m$.

Let the natural number $r\ge1$ satisfy
\[
\frac{r}{r+1}<\frac{m}{n}\le\frac{r+1}{r+2}.
\]
Let
\[
\mathcal{M}=V^{-2r}F^{2r+1}V^{-1}(0)(\Sigma).
\]
More precisely, we consider $V^{-2r}\{(F^{2r+1}(\mathcal{D}[V]^{(p^{2r+1})}))^{(p^{2r})}\}\subset\mathcal{D}$.
That this is a well-defined sub-bundle of $\mathcal{D},$ over any
EO stratum, and in particular over $S_{\fol},$ follows as before
from Proposition (3.2) in \cite{=00005BOo=00005D}. Hence the same
applies to its $\Sigma$-part, which is $\mathcal{M}.$

\medskip{}

\textbf{Claim. }Let $r$ and $\mathcal{M}$ be as above. Then, using
the notation of Proposition \ref{prop:DMfol}: (i) For any $x\in S_{\fol}$
\[
\mathcal{M}_{x}=Span_{k}\{e_{1},\dots,e_{2m}\}.
\]
(ii) Over $S_{\fol}$ we have $V(\mathcal{Q})\subset\mathcal{M}^{(p)}.$

\medskip{}

Part (i) will be proved in Lemma \ref{lem:DM_computation_M} below.
Part (ii) follows from Corollary \ref{cor:V_O(Q)_x at S_fol}. By
the corollary, the inclusion $V(\mathcal{Q})_{x}\subset\mathcal{M}_{x}^{(p)}$
holds between the \emph{fibers} of the two sub-bundles at any geometric
point $x\in S_{\fol}(k),$ and the base is reduced.

We can now apply a small variation on the case $2m\le n.$ From the
Claim we obtain
\[
V(\mathcal{Q})(\Spec(R_{\fol}))\subset\mathcal{M}^{(p)}(\Spec(R_{\fol}))=R_{\fol}\otimes_{k}\mathcal{M}_{x}^{(p)}=Span_{R_{\fol}}\{e_{1}^{(p)},\dots,e_{2m}^{(p)}\}.
\]
However, when $n<2m$ (\ref{eq:V-image}) implies
\[
V(f_{n+m+1-j}-\sum_{i=1}^{n-m}u_{ij}f_{n+m+1-i}-\sum_{\ell=n+1}^{n+m}v_{\ell j}f_{n+m+1-\ell})\equiv
\]
\[
-\sum_{i=1}^{n-m}u_{ij}e_{n+m+1-i}^{(p)}\mod R\otimes_{k}\mathcal{M}_{x}^{(p)}.
\]
Since $e_{n+m+1-i}^{(p)}$ ($1\le i\le n-m$) \emph{remain linearly
independent} modulo $\mathcal{M}_{x}^{(p)}$ we conclude that in $R_{\fol}$
we must have $\overline{u}_{ij}=0$. As before, this implies that
$R_{\fol}=R/(u_{ij})$, and concludes the proof of the theorem.
\end{proof}

\subsubsection{A Dieudonn\'e module computation}

To complete the proof of Theorem \ref{thm:Sfol_integral_submanifold}
when $n<2m$ we need to prove the following.
\begin{lem}
\label{lem:DM_computation_M}Let notation be as in Proposition \ref{prop:DMfol},
and let $r\ge1$ satisfy $r/(r+1)<m/n\le(r+1)/(r+2)$. Then
\[
V^{-2r}F^{2r+1}V^{-1}(0)=Span_{k}\{e_{[1,2m]},f_{[1,rn-(r-1)m]}\}.
\]
\end{lem}

\begin{proof}
Let $D_{0}(a,b)=Span_{k}\{e_{[1,a]},f_{[1,b]}\}.$ We first observe
that
\[
FD_{0}(a,b)=D_{0}(a^{-},b^{-}),\,\,\,\,V^{-1}D_{0}(a,b)=D_{0}(a^{+},b^{+})
\]
where
\[
a^{-}=\begin{cases}
\begin{array}{c}
b\\
m\\
b-m
\end{array} & \begin{array}{c}
0\le b\le m\\
m<b\le2m\\
2m<b\le n+m
\end{array}\end{cases}
\]
\[
b^{-}=\begin{cases}
\begin{array}{c}
0\\
a-n+m\\
m
\end{array} & \begin{array}{c}
0\le a\le n-m\\
n-m<a\le n\\
n<a\le n+m
\end{array}\end{cases}
\]

\[
a^{+}=\begin{cases}
\begin{array}{c}
n\\
b+n-m\\
n+m
\end{array} & \begin{array}{c}
0\le b\le m\\
m<b\le2m\\
2m<b\le n+m
\end{array}\end{cases}
\]
\[
b^{+}=\begin{cases}
\begin{array}{c}
a+m\\
n\\
a
\end{array} & \begin{array}{c}
0\le a\le n-m\\
n-m<a\le n\\
n<a\le n+m
\end{array}\end{cases}.
\]

To be precise, we should have written $FD_{0}(a,b)^{(p)}=D_{0}(a^{-},b^{-})$
etc., but from now on we omit the relevant Frobenius twists to simplify
the notation. We now compute, using these formulae, and leaving out
straightforward verifications:

\textbf{1. $V^{-1}(0)=D_{0}(n,m).$}

\textbf{2. }Let $0\le i\le r.$ One proves inductively that
\[
F^{2i}V^{-1}(0)=D_{0}(im-(i-1)n,\,(i+1)m-in)
\]
\[
F^{2i+1}V^{-1}(0)=D_{0}((i+1)m-in,\,(i+1)m-in).
\]

\textbf{3.} Let $1\le j\le r.$ Using induction on $j$ one shows
\[
V^{-2j+1}F^{2r+1}V^{-1}(0)=D_{0}(jn-(j-1)m,\,(r+3-j)m-(r+1-j)n)
\]
\[
V^{-2j}F^{2r+1}V^{-1}(0)=D_{0}((r+2-j)m-(r-j)n,\,jn-(j-1)m).
\]
The assumption that $r/(r+1)<m/n\le(r+1)/(r+2)$ is used repeatedly
in these computations. Putting $j=r$ proves the Lemma.
\end{proof}

\subsection{A conjecture of Andr\'e-Oort type}

Given a foliation in a real manifold, the celebrated theorem of Frobenius
says that integral subvarieties exist, and are unique, in sufficiently
small neighborhoods of any given point. Working in the algebraic category,
in characteristic $p$, one has to impose, in addition to the integrability
condition, also being closed under the $p$-power operation. Integral
subvarieties then exist in \emph{formal} neighborhoods, but are far
from being unique. For that purpose Ekedahl introduced in {[}Ek{]}
the notion of height $h$ foliations for any $h\ge1,$ a notion that
we do not discuss here, as our height 1 foliation does not seem to
extend to higher height foliations. Nor does the foliation lift to
characteristic 0 in any natural way; thus, the approach taken by Miyaoka
in {[}Mi{]} to deal with the same problem does not apply here.

Despite this lack of formal uniqueness, the global nature of our foliation
imposes a severe restriction on integral subvarieties. Thus, we dare
to make the following conjecture.
\begin{conjecture*}
The only integral subvarieties of the foliation $\mathcal{T}S^{+}$
in $S_{\sharp}$ are Shimura varieties of signature $(m,m)$ or $S_{\fol}.$
\end{conjecture*}


\begin{thebibliography}{De-Ill}
\bibitem[B-W]{=00005BB-W=00005D}O. B\"ultel, T. Wedhorn: Congruence
relations for Shimura varieties associated to some unitary groups,
J. Instit. Math. Jussieu \textbf{5} (2006), pp. 229-261.

\bibitem[De-Ill]{=00005BDe-Ill=00005D}P. Deligne, L. Illusie: Rel\`evements
modulo $p^{2}$ et d\'ecomposition du complexe de de Rham, Invent. Math.
\textbf{89}, 247-270 (1987). 

\bibitem[dS-G1]{=00005BdS-G1=00005D}E. de Shalit, E.Z. Goren: A theta
operator on Picard modular forms modulo an inert prime, Res. Math.
Sci. \textbf{3:}28 (2016), arXiv:1412.5494.

\bibitem[dS-G2]{=00005BdS-G2=00005D}E. de Shalit, E.Z. Goren: On
the bad reduction of certain $U(2,1)$ Shimura varieties, to appear in:
\emph{Geometry, Algebra, Number Theory and Their Information Technology Applications},
Toronto, Canada, June 2016 and Kozhikode, India, August 2016,
Edited by A. Akbary and S. Gun, 
Springer Proc. Math. and Stat.
arXiv:1703.05720.

\bibitem[dS-G3]{dS-G3}E. de Shalit, E.Z. Goren: Theta operators on unitary Shimura varieties, 
\emph{preprint} (2017), arXiv:1712.06969.

\bibitem[Eis]{=00005BEis=00005D}D. Eisenbud: \emph{Commutative Algebra
with a View Toward Algebraic Geometry}, GTM 150, Springer-Verlag,
New York, 1995.

\bibitem[Ek]{=00005BEk=00005D}T. Ekedahl: Foliations and inseparable
morphisms, in: \emph{Algebraic Geometry, Bowdin 1985,} Proc. Symp.
Pure Math. \textbf{46 }(2), 139-149 (1987).

\bibitem[GN]{GN} W. Goldring, M.-H. Nicole: The $\mu$-ordinary Hasse invariant of unitary Shimura varieties. J. Reine Angew. Math. 728 (2017), 137--151.

\bibitem[G\"o]{=00005BG=0000F6=00005D}U. G\"ortz: On the flatness of
models of certain Shimura varieties of PEL type, Math. Ann. \textbf{321}
(2001), pp. 689-727.

\bibitem[Gro]{=00005BGro=00005D}A. Grothendieck: Groupes de Barsotti-Tate
et cristaux de Dieudonn\'e, Univ. Montr\'eal, Montr\'eal, 1974.

\bibitem[Har]{=00005BHar=00005D}J. Harris: \emph{Algebraic Geometry:
A first course, }Springer, New York, 1992.

\bibitem[Ja]{=00005BJa=00005D}N. Jacobson: \emph{Lectures in abstract algebra. III. Theory of fields and Galois theory. }Second corrected printing. Graduate Texts in Mathematics, No. 32. Springer-Verlag, New York-Heidelberg, 1975. 

\bibitem[Ka]{=00005BKa=00005D}N. Katz: Nilpotent connections and
the monodromy theorem: Applications of a result of Turrittin, Publ.
Math IHES, \textbf{39} (1970), 175\textendash 232. 

\bibitem[Ko]{=00005BKo=00005D}R. E. Kottwitz: Points on some Shimura
varieties over finite fields, J. AMS \textbf{5} (1992), 373\textendash 444.

\bibitem[Ku]{=00005BKu=00005D}E. Kunz: Characterizations of regular
local rings of characteristic $p$. Amer. J. Math.,
\textbf{91}(3) (1969), 772\textendash 784.

\bibitem[Lan]{=00005BLan=00005D}K.-W. Lan: Arithmetic compactifications
of PEL-type Shimura varieties, London Mathematical Society Monographs
\textbf{36}, Princeton, 2013.

\bibitem[Mi]{=00005BMi=00005D}Y. Miyaoka: Deformation of a morphism
along a foliation, in: \emph{Algebraic Geometry, Bowdin 1985,} Proc.
Symp. Pure Math. \textbf{46 }(1), 245-268 (1987).

\bibitem[Mo1]{=00005BMo1=00005D}B. Moonen: Serre-Tate theory for
moduli spaces of PEL type, Ann. Scient. \'Ec. Norm. Sup., 4e s\'erie,
t. \textbf{37} (2004), pp. 223-269.

\bibitem[Mo2]{=00005BMo2=00005D}B. Moonen: Group schemes with additional
structures and Weyl group elements, in: \emph{Moduli of abelian varieties,
}C. Faber, G. van der Geer, F. Oort, eds., Progress in mathematics
195, Birkhauser, 2001, 255-298.

\bibitem[Oo]{=00005BOo=00005D}F. Oort: A stratification of a moduli
space of abelian varieties, in: \emph{Moduli of abelian varieties,
}C Faber, G. van der Geer, F. Oort, eds., Progress in mathematics
195, Birkhauser, 2001, 345-416.

\bibitem[P-Z]{=00005BP-Z=00005D}G. Pappas, X. Zhu: Local models of
Shimura varieties and a conjecture of Kottwitz, Invent. Math. \textbf{194}
(2013), 147-254

\bibitem[Ra-Zi]{=00005BRa-Zi=00005D}M. Rapoport, T. Zink: Period
Spaces for $p$-divisible Groups, Annals of Mathematics Studies \textbf{141},
Princeton University Press, Princeton (1996).

\bibitem[Ray]{=00005BRay=00005D}M. Raynaud: Sch\'emas en groupes de type $(p,...,p)$, Bull. Soc. Math. France, \textbf{102} (1974), 241-250.

\bibitem[Ru-Sh]{=00005BRu-Sh=00005D}A. N. Rudakov, I. R. Shafarevich:
Inseparable morphisms of algebraic surfaces, Izv. Akad. Nauk SSSR
Ser. Mat. \textbf{40 }(1976), pp. 1269\textendash 1307.

\bibitem[Va]{=00005BVa=00005D}A. Vasiu: Manin problems for Shimura
varieties of Hodge type, J. Ramanujan Math. Soc. \textbf{26} (2011),
31\textendash 84.

\bibitem[V-W]{=00005BV-W=00005D}E. Viehmann, T. Wedhorn: Ekedahl-Oort
and Newton strata for Shimura varieties of PEL type, Math. Ann. \textbf{356}
(2013), 1493\textendash 1550.

\bibitem[We]{=00005BWe=00005D}T. Wedhorn: The dimension of Oort strata
of Shimura varieties of PEL-type, in: \emph{Moduli of abelian varieties,
}C Faber, G. van der Geer, F. Oort, eds., Progress in mathematics
195, Birkhauser, 2001, 441-471.

\bibitem[Woo]{=00005BWoo=00005D}A. Wooding: The Ekedahl-Oort stratification
of unitary Shimura varieties, Ph.D. thesis, McGill University, 2016.
\end{thebibliography}
\end{document}